\documentclass{amsart}

\usepackage{amssymb}
\usepackage{amsmath}
\usepackage{amsfonts}
\usepackage{geometry}
\usepackage{bbm}
\usepackage{hyperref}
\usepackage{tikz}
\usetikzlibrary{matrix,arrows,decorations.pathmorphing}

\newcounter{cprop}[section]

\newtheorem{theorem}[cprop]{Theorem}

\theoremstyle{plain}

\newtheorem{lemma}[cprop]{Lemma}

\numberwithin{equation}{section}

\theoremstyle{definition}
\newtheorem{definition}[cprop]{Definition}

\theoremstyle{remark}
\newtheorem{remark}[cprop]{Remark}

\newcommand{\E}{\mathbb{E}}

\newcommand{\R}{\mathbb{R}}
\newcommand{\N}{\mathbb{N}}

\newcommand{\vertiii}[1]{{\left\vert\kern-0.25ex\left\vert\kern-0.25ex\left\vert #1 
    \right\vert\kern-0.25ex\right\vert\kern-0.25ex\right\vert}}

\begin{document}
\title{Rough differential equations with unbounded drift term}

\author{S. Riedel}
\address{Sebastian Riedel \\
Institut f\"ur Mathematik, Technische Universit\"at Berlin, Germany}
\email{riedel@math.tu-berlin.de}

\author{M. Scheutzow}
\address{Michael Scheutzow \\
Institut f\"ur Mathematik, Technische Universit\"at Berlin, Germany}
\email{ms@math.tu-berlin.de}

\keywords{controlled ordinary differential equations, rough paths, stochastic differential equations}

\subjclass[2010]{34A34, 34F05, 60G15, 60H10}

\begin{abstract}
  We study controlled differential equations driven by a rough path (in the sense of T. Lyons) with an additional, possibly unbounded drift term. We show that the equation induces a solution flow if the drift grows at most linearly. Furthermore, we show that the semiflow exists assuming only appropriate one-sided growth conditions. We provide bounds for both the flow and the semiflow. Applied to stochastic analysis, our results imply \emph{strong completeness} and the existence of a stochastic (semi)flow for a large class of stochastic differential equations. If the driving process is Gaussian, we can further deduce (essentially) sharp tail estimates for the (semi)flow and a Freidlin-Wentzell-type large deviation result.
\end{abstract}

\maketitle

\section*{Introduction}

T. Lyons' theory of \emph{rough paths} can be used to solve controlled ordinary differential equations (ODE) of the form
\begin{align}\label{eqn:controlled_ODE_illposed}
  \begin{split}
 dy &= b(y)\, dt + \sum_{i=1}^d \sigma_i(y)\, dx^i_t; \quad t \in [0,T] \\
 y_0 &= \xi \in \R^m
  \end{split}
\end{align}
for vector fields $b, \sigma_1, \ldots, \sigma_d \colon \R^m \to \R^m$ and non-differentiable, $1/p$-H\"older continuous paths \\ $x \colon [0,T] \to \R^d$. However, one of Lyons' key insights was that the equation \eqref{eqn:controlled_ODE_illposed} as it stands is ill-posed\footnote{More precisely, Lyons showed that the map assigning to each smooth path $x$ the solution $y$ to the ordinary differential equation \eqref{eqn:controlled_ODE_illposed} is not closable in the space of $p$-variation or $1/p$-H\"older continuous paths.} in the case of $p \geq 2$. Instead, one has to enhance the path $x \colon [0,T] \to \R^d$ with additional information (which can be interpreted as its iterated integrals) to a path $\mathbf{x}$ taking values in a larger space. Defining a suitable ($p$-variation or H\"older-type) topology on this space of paths allows to solve the corresponding ``lifted'' equation
\begin{align}\label{eqn:controlled_ODE_wellposed}
  \begin{split}
 dy &= b(y)\, dt + \sigma(y)\, d\mathbf{x}_t; \quad t \in [0,T] \\
 y_0 &= \xi \in \R^m
  \end{split}
\end{align}
uniquely in the way that the solution map (also called \emph{It\=o-Lyons map}) $\mathbf{x} \mapsto y$ is continuous. This paves way to a genuine \emph{pathwise} stochastic calculus for a huge class of (not-necessarily martingale-type) driving signals (cf. e.g. \cite[Chapter 13 - 20]{FV10} and the references therein). Rough paths theory is now well-established, and since Lyons' seminal article \cite{Lyo98}, several monographs have appeared (cf. \cite{LQ02, LCL07, FV10, FH14}) which expose the theory and its various applications. Let us also briefly mention that rough paths ideas were used by M. Hairer to solve stochastic partial differential equations (SPDE) like the KPZ-equation (\cite{Hai13}) and form an important part in his theory of \emph{regularity structures} 
(cf. \cite{Hai14} and \cite{FH14} where the link between rough paths and regularity structures is explained). 

In the present work, we aim to solve \eqref{eqn:controlled_ODE_wellposed} for a general, possibly unbounded drift term $b$ while we assume $\sigma$ to be bounded and sufficiently smooth. In the literature about rough paths, a convenient way to take care of the drift part is to regard $t \mapsto t$ as an additional (smooth) component of the rough path $\mathbf{x}$, and $b$ as another component of $\sigma$ (cf. e.g. \cite[Exercise 8.15]{FH14}). However, this yields unnecessary smoothness assumptions,
and allowing $b$ to be unbounded leads to the study of general unbounded vector fields for rough differential equations (which is a delicate topic, cf. \cite{Lej12} for a discussion). Maybe more important, the bounds for the solution $y$ which are available in this case (cf. e.g. \cite[Exercise 10.56]{FV10}) are bounds which grow exponentially in the rough path norm of $\mathbf{x}$, whereas bounded diffusion vector fields should yield polynomial bounds. The main theorems in the present paper (Theorem \ref{thm:flow_lin_growth} and Theorem \ref{thm:main_semiflow}) provide exactly the bounds expected.

A rough differential equation can be seen as a special case of a non-autonomous ordinary differential equation. Therefore, it should not come as a big surprise that such equations naturally induce continuous two parameter flows\footnote{In fact, in \cite{Bai15}, the flow is even the central object of interest and it is constructed directly, skipping the intermediate step of defining the solution to \eqref{eqn:controlled_ODE_wellposed} for a fixed initial datum $\xi$ first.} on the state space $\R^m$ (at least if all vector fields are bounded, cf. \cite{LQ98}, \cite[Section 11.2]{FV10}, \cite[Section 8.9]{FH14}). Note that this immediately implies that a stochastic differential equation (SDE) induces a stochastic flow provided the driving process has sample paths in a rough paths space (which is the case, for instance, for a Brownian motion). In particular, the SDE is \emph{strongly complete} which means that it can be solved globally on a set of full measure which does not depend on the initial condition. Note that an SDE may lack strong completeness while possessing strong solutions (in the It\=o-sense) for any initial condition. Indeed, this is even possible for $b \equiv 0$ and $\sigma$ bounded and $\mathcal{C}^{\infty}$ (but with unbounded derivatives), cf. \cite{LS11}. However, using a pathwise calculus (like rough paths theory), strong completeness is immediate.

We are interested in proving the existence of a (semi)flow induced by \eqref{eqn:controlled_ODE_wellposed} for an unbounded drift $b$. In Section \ref{sec:RDE_flows}, we 
first discuss the case of a proper flow, i.e. the case when \eqref{eqn:controlled_ODE_wellposed} can be solved forward and backward in time. In this case, it is natural to assume that $b$ should be locally Lipschitz continuous with linear growth, and in Theorem \ref{thm:flow_lin_growth} we prove the existence of the flow under these assumptions and provide quantitative bounds. More interesting might be the case when we can only expect to solve \eqref{eqn:controlled_ODE_wellposed} in one time direction, say forward in time (a typical example would be $b(\xi) = -\xi|\xi|^2$). In these situations, the best we can hope for is to prove existence of a \emph{semi}flow induced by \eqref{eqn:controlled_ODE_wellposed}. A classical condition to impose (both in the theory of ODE and SDE) is the one-sided growth condition
\begin{align}\label{eqn:one_sided_growth_intro}
 \langle b(\xi),\xi \rangle \leq C(1 + |\xi|^2)\quad \text{for all } \xi \in \R^m
\end{align}
together with a (one-sided) local Lipschitz condition. In the context of SDE driven by a $d$-dimensional Brownian motion, strong global existence and uniqueness under condition \eqref{eqn:one_sided_growth_intro} was proven in \cite{PR07}. Recently, one of the authors showed in \cite{SS16} the existence of a semiflow even for infinitely many Brownian motions under slightly stronger assumptions. Interestingly, if $m \geq 2$, imposing only \eqref{eqn:one_sided_growth_intro} is not enough to imply non-explosion for solutions to \eqref{eqn:controlled_ODE_wellposed} on a pathwise level. Indeed, a counterexample can be found in \cite[p. 43]{CHJ13} already in the case of ``additive noise'', i.e. for $\sigma$ being constant. There, the authors define an explicit vector field $b \colon \R^2 \to \R^2$ with a strong (cubic is enough) growth  in the \emph{tangential} direction only. Then, they construct an (even smooth!) path $x \colon [0, \infty) \to \R^2$ and show that the solution to \eqref{eqn:controlled_ODE_illposed}
 explodes in finite time. This suggests the need to impose an additional condition on $b$ which controls the growth in tangential direction. In the case of additive noise, it was shown in \cite{SS16} that non-explosion can be assured even for quadratic tangential growth. In this work, we impose a linear growth of the form
\begin{align}\label{eqn:tangential_growth_intro}
 \left| b(\xi) - \frac{\langle b(\xi) ,\xi \rangle \xi}{|\xi|^2} \right| \leq C(1 + |\xi|)\quad \text{for all } \xi \in \R^m.
\end{align}
Our second main result (Theorem \ref{thm:main_semiflow}) states that under the two growth conditions \eqref{eqn:one_sided_growth_intro} and \eqref{eqn:tangential_growth_intro} and a suitable local Lipschitz condition, the semiflow to \eqref{eqn:controlled_ODE_wellposed} exists. Moreover, we provide quantitative bounds which are similar to those derived for the flow in Theorem \ref{thm:flow_lin_growth}.

To illustrate our results, let us discuss some applications in stochastic analysis. 

\begin{theorem}\label{thm:appl_intro}
 Let $\sigma = (\sigma_1,\ldots,\sigma_d)$ be a collection of infinitely often differentiable vector fields on $\R^m$ where $\sigma$ and all its derivatives are bounded. Consider the stochastic differential equation
 \begin{align}\label{eqn:SDE_intro}
  dY &= (Y - |Y|^2 Y)\, dt + \sum_{i = 1}^d \sigma_i(Y)\, \circ dX^i_t(\omega); \quad t \in [0,T] \\
  Y_0 &= \xi \in \R^m 
 \end{align}
 where $X \colon [0,T] \to \R^d$ is a continuous stochastic process which can be enhanced to a process with values in the space of weakly geometric rough paths on set of full measure (this can be a semimartingale, a Gaussian process or a Markov processes, cf. \cite{FV10} for a list of examples). The equation \eqref{eqn:SDE_intro} is understood in rough paths sense or in Stratonovich sense in case of $X$ being a semimartingale and the lift is defined as in \cite[Chapter 14]{FV10}.
 
 Then the following holds:
 \begin{itemize}
  \item[(i)] The SDE \eqref{eqn:SDE_intro} is strongly complete and induces a continuous stochastic semiflow.
  \item[(ii)] If $X \colon [0,T] \to \R^d$ is a centered Gaussian process with covariance of finite $(1,\rho)$-variation for some $\rho \in [1,2)$ (cf. \cite{FGGR16} for the precise definition), the random variable $\|Y \|_{\infty;[0,T]}^{1/\rho}$ has Gaussian tails, i.e. there is a $\delta > 0$ such that
  \begin{align*}
   \E\left[ \exp(\delta \|Y \|_{\infty;[0,T]}^{2/\rho}) \right] < \infty.
  \end{align*}
  Moreover, the random variables $(Y^{\varepsilon}_t\, :\, 0\leq t \leq T)$, $Y^{\varepsilon}$ being the solution to \eqref{eqn:SDE_intro} when we replace $dX^i_t$ by $\varepsilon dX^i_t$, satisfy a Freidlin-Wentzell-type large deviation principle in the topology of uniform convergene (cf. \cite[Proposition 19.14]{FV10} for the precise formulation).

 \end{itemize}

\end{theorem}

\begin{proof}
 The vector field $b(\xi) = \xi - |\xi|^2 \xi$ satisfies the conditions of Theorem \ref{thm:main_semiflow}. Therefore, \eqref{eqn:SDE_intro} can be solved pathwise which implies (i). In case of $X(\omega)$ being a Gaussian process with lift $\mathbf{X}(\omega)$, the quantity $N_1(\mathbf{X}(\omega))^{1/\rho}$ is a random variable with Gaussian tails; cf. \cite[Theorem 1.1]{FGGR16} and \cite[Theorem 11.13]{FH14}, and the bound \eqref{eqn:sup_norm_bound_semiflow} implies the tail estimate in (ii). The large deviation result follows by a Schilder-type large deviation result for $\mathbf{X}$ (cf. \cite[Theorem 15.55]{FV10}) and the \emph{contraction principle} which can be used since $\mathbf{X}(\omega) \mapsto Y(\omega)$ is continuous by Theorem \ref{thm:main_semiflow}.
\end{proof}

Let us remark that
\begin{itemize}
 \item[(i)] the smoothness assumptions for $\sigma$ can be relaxed and are linked to the ``roughness'' of the trajectories of $X$, cf. Theorem \ref{thm:main_semiflow}.
 \item[(ii)] The uniform norm in Theorem \ref{thm:appl_intro} can be replaced by the $p$-variation norm for sufficiently large $p$ (where $p$ depends on the rough path trajectories). 
 \item[(iii)] The large deviation principle also holds in $p$-variation topology (again, for $p$ large enough).
 \item[(iv)] Fractional Brownian motion with Hurst parameter $H$ falls into the framework of Theorem \ref{thm:appl_intro} with $H = 1/(2\rho)$ (other examples of Gaussian processes may be found in \cite{FGGR16}).
\end{itemize}

The article is organized as follows: In Section \ref{sec:notation}, we quickly recall some basic facts about rough paths and explain some notation. Section \ref{sec:prelim} introduces the flow decomposition (our main technique for proving our results) and some facts about flows induced by rough differential equations with bounded coefficients are proved. In Section \ref{sec:RDE_flows}, we prove our main result for $b$ having linear growth, cf. Theorem \ref{thm:flow_lin_growth}. Finally, in Section \ref{sec:RDE_semiflows} we study the case where $b$ is assumed to satisfy only one-sided growth conditions. Our main results here are formulated in Theorem \ref{thm:main_semiflow}.

\section{Notation, elements of rough path theory}\label{sec:notation}

We will now very briefly recall the elements of rough paths theory used in
this paper. For more details we refer to \cite{FV10}, \cite{LCL07}, \cite{LQ02} or \cite{FH14}. Our notation coincides with the one used in  \cite{FV10}.

Let $T^N(\R^d) =\R %
\oplus \R^d \oplus (\R^d \otimes \R^d) \oplus \ldots \oplus (\R^d)^{\otimes
N}$, be the truncated step-$N$ tensor algebra, $N \geq 1$. We are concerned with $T^N(\R^d)$-valued
paths, as naturally given by iterated integrations of $\R^d$-valued smooth paths (``lifted smooth paths").
The projection of such a path $\mathbf{x}$ on the first level is an $\R^d$-valued path and will be denoted by $\pi_1(%
\mathbf{x})$, the projection to $k$th level is denoted by $\pi_k$. Lifted smooth paths actually take values in $G^N(\R^d) \subset T^N(\R^d)$, where $(G^N(\R^d),\otimes)$ denotes the
free step-$N$ nilpotent Lie group with $d$ generators (cf. \cite[Theorem and Definition 7.30]{FV10}). The group structure allows to define natural increments $\mathbf{x}_{s,t}\equiv \mathbf{x}_{s}^{-1}\otimes \mathbf{x}_{t}$, $s,t \in \R$, for paths $\mathbf{x}$ taking values in $G^N(\R^d)$. The (left-invariant) Carnot-Caratheodory metric turns $(G^N(\R^d),d)$ into a metric space (\cite[Section 7.5.4]{FV10}).

Fix some time interval $[0,T]$. For $p \geq 1$ and $[s,t] \subseteq [0,T]$, we will use the $p$-variation and $1/p$-H\"older ``norm''

\begin{eqnarray*}
\left\Vert \mathbf{x}\right\Vert _{p\text{-var;}\left[ s,t\right] }
&=&\sup_{\left( t_{i}\right) \subset \left[ s,t\right] }\left(
\sum_{i}d\left( \mathbf{x}_{t_{i}},\mathbf{x}_{t_{i+1}}\right) ^{p}\right)
^{1/p},  \label{DefinitionDoubleBarPvarOnGroup} \\
\left\Vert \mathbf{x}\right\Vert _{1/p\text{-H\"{o}l;}\left[ s,t\right]
}  &=&   \sup_{s\leq u<v\leq t}\frac{d\left( \mathbf{x}_{u},\mathbf{x}_{v}\right) }{%
\left\vert v-u\right\vert ^{1/p}},
  \notag
\end{eqnarray*}%
and distances

\begin{eqnarray*}
d_{p%
\text{-var;}\left[ s,t\right] }(\mathbf{x},\mathbf{y}) &= & \left(
\sup_{(t_i) \subset [s,t] }\sum_{i}d\left( \mathbf{x}_{t_{i},t_{i+1}},\mathbf{y}%
_{t_{i},t_{i+1}}\right) ^{p}\right) ^{1/p}, \\
d_{1/p\text{-H\"{o}l;}\left[ s,t\right] }(\mathbf{x},\mathbf{y})
 &= & \sup_{s\leq u<v\leq t}\frac{d\left( \mathbf{x}_{u,v},\mathbf{y}%
_{u,v}\right) }{\left\vert v-u\right\vert ^{1/p}}.
  \notag
\end{eqnarray*}
To simplify notation, we will occasionally write $\| \cdot \|_{p-\text{var}} := \| \cdot \|_{p-\text{var};[0,T]}$ and $d_{p-\text{var}} := d_{p-\text{var};[0,T]}$; similar for the H\"older case.

A \emph{weak geometric $p$-rough path} is a continuous path with finite $p$-variation which takes values in $G^{\lfloor p \rfloor}(\R^d)$. A \emph{weak geometric $1/p$-H\"older rough path} is a weak geometric $p$-rough path for which its $1/p$-H\"older norm is finite (cf. \cite[Definition 9.15]{FV10}).

Recall that a \emph{control function} $\omega $ is a continuous function $\omega \colon \{ \, 0 \leq s \leq t \leq T\} \to [0,\infty)$ such that $\omega(t,t) = 0$ for every $t \in [0,T]$ and which is \emph{superadditive}, i.e.
\begin{align*}
  \omega(s,t) + \omega(t,u) \leq \omega(s,u)
\end{align*}
 holds for every $s \leq t \leq u$. Typical examples of control functions are $(s,t) \mapsto |t - s|$ and $(s,t) \mapsto \|\mathbf{x}\|_{p-\text{var};[s,t]}^p$ where $\mathbf{x}$ is any $p$-rough path. We say that \emph{$\omega$ controls the $p$-variation of $\mathbf{x}$} if $d(\mathbf{x}_s,\mathbf{x}_t)^p \leq \omega(s,t)$ holds for every $s \leq t$. Note that this is equivalent to say that $\|\mathbf{x} \|_{p-\text{var};[s,t]}^p \leq \omega(s,t)$ holds for every $s<t$. If $\mathbf{x}$ has finite $p$-variation, its $p$-variation is controlled by $\omega(s,t) = \|\mathbf{x} \|_{p-\text{var};[s,t]}^p$.
 
 For a control $\omega$ and some $\delta > 0$, we define a sequence $(\tau_n)$ as follows: set $\tau_0 := 0$ and
  \begin{align*}
   \tau_{n+1} := \inf\{ u\, :\, \omega (\tau_n,u) \geq \delta,\, \tau_n < u \leq T \} \wedge T.
  \end{align*}
  Then we define
  \begin{align*}
    N_\delta(\omega) := \sup \{ n \in \N\,:\, \tau_n < T \}.
  \end{align*}
  From superadditivity of $\omega$, $\delta N_\delta(\omega) \leq \omega(0,T) < \infty$. If $\omega(s,t) = \| \mathbf{x} \|_{p-\text{var};[s,t]}^p$ for some rough path $\mathbf{x}$, we will also write $N_\delta(\mathbf{x})$ for $N_\delta(\omega)$. The quantity $N_\delta(\omega)$ first appeared in \cite{CLL13} where the authors observed that $N_\delta(\mathbf{X})$ has significantly better integrability properties than $\| \mathbf{X} \|_{p-\text{var}}^p$ when $\mathbf{X}$ is the lift of a Gaussian stochastic process, cf. also \cite[Section 11.2]{FH14}.

A collection of vector fields $\sigma = (\sigma_1,\ldots,\sigma_d)$ on $\R^m$ is called \emph{$\gamma$-Lipschitz (in the sense of E. Stein)} for $\gamma > 0$, denoted $\sigma \in \operatorname{Lip}^{\gamma}$, if all $\sigma_i$ are $\lfloor \gamma \rfloor$-times continuously differentiable, the vector fields and all derivatives up to order $\lfloor \gamma \rfloor$ are bounded, and the $\lfloor \gamma \rfloor$-th derivatives are $(\gamma - \lfloor \gamma \rfloor)$-H\"older continuous. If $\gamma$ is an integer, this means that the $(\gamma - 1)$-th derivatives are Lipschitz continuous (cf. \cite[Definition 10.2]{FV10}). The smallest constant which bounds the supremum norm of the vector fields, its derivatives and the H\"older norm of the $\lfloor \gamma \rfloor$-th derivatives is denoted by $|\sigma|_{\operatorname{Lip}^{\gamma}}$.

We will be interested in rough differential equations of the form
\begin{align}\label{eqn:rde_with_drift}
 dy &= b(y)\, dt + \sigma(y)\, d\mathbf{x}; \quad t \in [0,T]
\end{align}
where $\mathbf{x}$ is a $p$-rough path in $G^{\lfloor p \rfloor}(\R^d)$, the solution $y$ is a continuous path in $\R^m$, $b$ and $\sigma = (\sigma_1,\ldots, \sigma_d)$ are vector fields in $\R^m$. In the following, we recall the definition of a solution to \eqref{eqn:rde_with_drift} in the sense of Friz-Victoir \cite[Definition 10.17]{FV10}:

\begin{definition}
 Let $\mathbf{x}$ be a $p$-rough path. A path $y \colon [0,T] \to \R^m$ is called a \emph{solution to \eqref{eqn:rde_with_drift} with initial condition $y_0 = \xi \in \R^m$ in the sense of Friz-Victoir} if the following holds: 
 \begin{itemize}
  \item[(i)] $y_0 = \xi$.
  \item[(ii)] There exists a sequence $(x^n)$ of continuous paths in $\R^d$ with finite variation such that the lifted paths $\mathbf{x}^n$ satisfy
 \begin{align}\label{eqn:approx_paths}
  \sup_{n \in \N} \| \mathbf{x}^n \|_{p-\text{var}} < \infty \qquad \text{and} \qquad \lim_{n \to \infty} \sup_{0 \leq s < t \leq T} d(\mathbf{x}^n_{s,t}, \mathbf{x}_{s,t}) = 0,
 \end{align}
 and that there are solutions $y^n$ to the ordinary differential equations
\begin{align*}
  \begin{split}
 d y^n_t &=  b(y^n_t)\, dt + \sigma (y^n_t)\, d x^{n}_t, \quad t \in [0,T] \\
 y^n_0 &= \xi
  \end{split}
\end{align*}
which converge uniformly to $y$ on $[0,T]$ as $n \to \infty$.
 \end{itemize}

\end{definition}

Note that by \cite[Proposition 8.12]{FV10}, for any given weak geometric $p$-rough path $\mathbf{x}$ we can find a sequence of continuous paths $(x^n)$ with finite variation for which the canonical lifts satisfy \eqref{eqn:approx_paths}.

Let $M$ be a set. A mapping $\phi \colon [0,T] \times [0,T] \times M \to M$ is called a \emph{flow} on $M$ if
\begin{itemize}
 \item[(i)] $\phi(t,t,\xi) = \xi$ for every $t \in [0,T]$ and every $\xi \in M$ and
 \item[(ii)] $\phi(s,t,\xi) = \phi(u,t,\phi(s,u,\xi))$ for every $s,u,t \in [0,T]$ and $\xi \in M$
\end{itemize}
hold. If property (ii) only holds for $s \leq u \leq t$, $\phi$ is called a \emph{semiflow}. If $M$ is a topological space and $\phi$ is jointly continuous, we speak of a \emph{continuous} (semi)flow. If $M = \R^m$, $\phi(s,t,\cdot)$ is differentiable for every $s,t \in [0,T]$ (resp. $s \leq t$) and if the derivative is jointly continuous, $\phi$ is called a \emph{continuously differentiable} (semi)flow.

For vectors $v,w \in \R^m$, $|v|$ will denote the standard $l^2$-norm and $\langle v,w \rangle$ their Euclidean scalar product. For matrices $A \in \R^{m \times m}$, $|A|$ denotes the usual operator norm for a linear map.

\section{Preliminaries}\label{sec:prelim}

Let $x \colon [0,T] \to \R^d$ be smooth and $\sigma \colon \R^m \to \operatorname{Lin}(\R^d, \R^m)$ be smooth and bounded with bounded derivatives. Let $\psi \colon [0,T] \times [0,T] \times \R^m \to \R^m$ be the solution flow to the (non-autonomous) ordinary differential equation
 \begin{align}\label{eqn:controlled_ODE_flow}
  \dot{y}_t = \sigma(y_t) \dot{x}_t.
 \end{align}
 For given $b \colon \R^m \to \R^m$, assume that we can make sense of the ordinary differential equation
 \begin{align}\label{eqn:flow_decomp_ODE}
  \begin{split}
  \dot{z}_u &=  (D_{\xi} \psi(s,u,\xi)\vert_{\xi = z_u})^{-1} b(\psi(s,u,z_u)); \quad u \in [0,T]\\
  z_s &= \xi
  \end{split}
 \end{align}
 for any $s,t \in [0,T[$ and any $\xi \in \R^m$. Let $\chi_s(t,\xi)$ denote the value of the solution to \eqref{eqn:controlled_ODE_flow} at time point $t$. Then an easy application of the chain rule shows that $\phi(s,t,\xi) := \psi(s,t,\chi_s(t,\xi))$ coincides with the solution flow to the equation
 \begin{align*}
  \dot{y}_t = b(y_t) + \sigma(y_t) \dot{x}_t.
 \end{align*}
In the cases we will consider, solutions to \eqref{eqn:flow_decomp_ODE} will only exist on small time intervals and possibly only forward in time. Taking this into account, we make the following definition:

\begin{definition}\label{def:sol_flow_rde_with_drift}
 Let 
 \begin{align*}
    \mathcal{I} \subseteq \{[s,t]\, :\, s \leq t,\, s,t \in [0,T] \}
 \end{align*}
 be a subset of the set of all intervals contained in $[0,T]$ for which there exists a finite subset $\mathcal{I}_0 \subseteq \mathcal{I}$ such that  
 \begin{align*}
  \bigcup_{[u,v] \in \mathcal{I}_0} [u,v] = [0,T]
 \end{align*}
 holds and for which $[u,v] \subset [s,t]$ and $[s,t] \in \mathcal{I}$ implies that $[u,v] \in \mathcal{I}$. Let $\mathbf{x}$ be a weak geometric $p$-rough path with values in $G^{\lfloor p \rfloor}(\R^d)$, $p \in [1,\infty)$, $b$ a vector field and $\sigma = (\sigma_1,\ldots,\sigma_d)$ a collection of vector fields on $\R^m$. Assume that the rough differential equation
 \begin{align}\label{eqn:rough_DE_flow}
  d y_t = \sigma(y_t)\, d \mathbf{x}_t; \quad t \in [0,T]
 \end{align}
 induces a continuously differentiable solution flow $\psi^{\mathbf{x}} \colon [0,T] \times [0,T] \times \R^m \to \R^m$ and that the ordinary differential equation
  \begin{align}\label{eqn:ODE_flow}
   \begin{split}
  \dot{z}_u &=  (D_{\xi} \psi^{\mathbf{x}}(v,u,\xi)\vert_{\xi = z_u})^{-1} b(\psi^{\mathbf{x}}(v,u,z_u)); \quad u \in [s,t]\\
  z_v &= \xi
    \end{split}
 \end{align}
 has a unique solution (forward and backward in time) for every $[s,t] \in \mathcal{I}$, $v \in [s,t]$ and $\xi \in \R^m$. We denote this solution by $[s,t] \ni u \mapsto \chi_v^{\mathbf{x}}(u,\xi)$. 
 Set 
 \begin{align*}
  \phi^{\mathbf{x}}(s,t,\xi) &= \psi^{\mathbf{x}}(s,t,\chi_s^{\mathbf{x}}(t,\xi)) \quad \text{and}\\
  \phi^{\mathbf{x}}(t,s,\xi) &= \psi^{\mathbf{x}}(t,s,\chi_t^{\mathbf{x}}(s,\xi))
 \end{align*}
  for $[s,t] \in \mathcal{I}$ and
 \begin{align*}
  \phi^{\mathbf{x}}(s,t,\xi) &:= \phi^{\mathbf{x}}(t_{n-1},t_n,\cdot) \circ \cdots \circ \phi^{\mathbf{x}}(t_0,t_1,\xi) \quad \text{and}\\
  \phi^{\mathbf{x}}(t,s,\xi) &:= \phi^{\mathbf{x}}(t_1,t_0,\cdot) \circ \cdots \circ \phi^{\mathbf{x}}(t_n,t_{n-1},\xi)
 \end{align*}
 for arbitrary $s \leq t$ where $s = t_0 < \ldots < t_n = t$ and $[t_i,t_{i+1}] \in \mathcal{I}$ for every $i = 0, \ldots, n-1$. Then we call the map $\phi^{\mathbf{x}} \colon [0,T] \times [0,T] \times \R^m \to \R^m$ the \emph{solution flow to \eqref{eqn:rde_with_drift}}. If the solution to \eqref{eqn:rough_DE_flow} exists only forward in time, we define $\phi(s,t,\xi)$ for $s \leq t$ and $\xi \in \R^m$ as above and call it the \emph{solution semiflow to \eqref{eqn:rde_with_drift}}.
 
\end{definition}

To simplify notation, we will sometimes drop the upper index $\mathbf{x}$ and just write $\phi$, $\psi$ and $\chi$.

We have to check that $\phi$ is well defined and does not depend on the choice of $\mathcal{I}$. This is done in the next lemma.

\begin{lemma}\label{lemma:flow_well_defined}
 Under the conditions stated in Definition \ref{def:sol_flow_rde_with_drift}, $\phi \colon [0,T] \times [0,T] \times \R^m \to \R^m$ is well defined, does not depend on the choice of $\mathcal{I}$ and satisfies the (semi-)flow property.
\end{lemma}

\begin{proof}
 We first check that $\phi$ is well defined as a semiflow. Note that it is enough to prove that
\begin{align*}
 \phi(u_1,t,\phi(s,u_1,\xi)) = \phi(u_2,t,\phi(s,u_2,\xi))
\end{align*}
holds for $s \leq u_1 \leq u_2 \leq t$, $[s,u_1]$, $[s,u_2]$, $[u_1,t]$, $[u_2,t] \in \mathcal{I}$ and $\xi \in \R^m$. From the flow property of $\psi$, this is equivalent to
\begin{align*}
 \psi(u_2,t,\psi(u_1,u_2,\chi_{u_1}(t,\phi(s,u_1,\xi)))) = \psi(u_2,t,\chi_{u_2}(t,\phi(s,u_2,\xi)).
\end{align*}
Therefore, it is enough to check that
\begin{align*}
 \psi(u_1,u_2,\chi_{u_1}(t,\phi(s,u_1,\xi))) = \chi_{u_2}(t,\phi(s,u_2,\xi).
\end{align*}
We do this by showing that both objects, seen as functions in $t$, solve the differential equation
  \begin{align}
    \dot{z}_u &=  (D_{\zeta} \psi(u_2,u,\zeta)\vert_{\xi = z_u})^{-1} b(\psi(u_2,u,z_u)); \quad u \in [u_2,t] \label{eqn:ode_well_defined} \\
  z_{u_2} &= \phi(s,u_2,\xi).
 \end{align}
 The function $ u \mapsto \chi_{u_2}(u,\phi(s,u_2,\xi)$ solves this equation by definition, and it remains to show that also the function on the left hand side solves the same equation. We first check that it has the same initial condition. We need to show that
 \begin{align*}
  \psi(u_1,u_2,\chi_{u_1}(u_2,\phi(s,u_1,\xi))) = \phi(s,u_2,\xi) = \psi(s,u_2,\chi_s(u_2,\xi)) = \psi(u_1,u_2,\psi(s,u_1, \chi_s(u_2,\xi))).
 \end{align*}
 Thus, it is enough to establish the identity
 \begin{align*}
  \chi_{u_1}(u_2,\phi(s,u_1,\xi)) = \psi(s,u_1, \chi_s(u_2,\xi)).
 \end{align*}
 This is done by showing that both expressions, seen as functions in $u_2$, solve the differential equation
  \begin{align*}
    \dot{z}_u &=  (D_{\zeta} \psi(u_1,u,\zeta)\vert_{\zeta = z_u})^{-1} b(\psi(u_1,u,z_u)); \quad u \in [u_1,u_2]\\
  z_{u_1} &= \phi(s,u_1,\xi).
 \end{align*}
 The function on the left hand side solves the equation by definition. For the function on the right hand side, we have
 \begin{align*}
  \psi(s,u_1, \chi_s(u_2,\xi)) \vert_{u_2 = u_1} = \psi(s,u_1, \chi_s(u_1,\xi)) = \phi(s,u_1,\xi),
 \end{align*}
 thus the initial condition is satisfied. Differentiating this function, using the chain rule, gives
 \begin{align*}
  \frac{d}{du} \psi(s,u_1, \chi_s(u,\xi)) &= D_{\zeta} \psi(s,u_1,\zeta)\vert_{\zeta = \chi_s(u,\xi)} \frac{d}{du} \chi_s(u,\xi) \\
  &= D_{\zeta} \psi(s,u_1,\zeta)\vert_{\zeta = \chi_s(u,\xi)} (D_{\zeta} \psi(s,u,\zeta) \vert_{\zeta = \chi_s(u,\xi)})^{-1} b(\psi(s,u,\chi_s(u,\xi))) \\
  &= D_{\zeta} \psi(s,u_1,\zeta)\vert_{\zeta = \chi_s(u,\xi)} (D_{\zeta} \psi(s,u,\zeta) \vert_{\zeta = \chi_s(u,\xi)})^{-1} b(\psi(u_1,u,\psi(s,u_1,\chi_s(u,\xi)))).
 \end{align*}
 It remains to show that
 \begin{align*}
  D_{\zeta} \psi(s,u_1,\zeta)\vert_{\zeta = \chi_s(u,\xi)} (D_{\zeta} \psi(s,u,\zeta) \vert_{\zeta = \chi_s(u,\xi)})^{-1} = (D_{\zeta} \psi(u_1,u,\zeta)\vert_{\zeta = \psi(s,u_1, \chi_s(u,\xi))})^{-1}.
 \end{align*}
 This identity follows by differentiating both sides of $\psi(s,u,\theta) = \psi(u_1,u,\psi(s,u_1,\theta))$ with respect to $\theta$ and substituting $\theta = \chi_s(u,\xi)$. Going back our proof, we see that we still have to show that $u \mapsto \psi(u_1,u_2,\chi_{u_1}(u,\phi(s,u_1,\xi)))$ satisfies \eqref{eqn:ode_well_defined}, but this is done exactly as above. It follows that $\phi$ is indeed well defined. The semiflow property follows by definition. 
 
 Next, we show that $\phi$ does not depend on $\mathcal{I}$. Let $\phi^1$ and $\phi^2$ be two semiflows associated to $\mathcal{I}^1$ resp. $\mathcal{I}^2$. Note that $\mathcal{I} := \mathcal{I}^1 \cup \mathcal{I}^2$ satisfies the same conditions as $\mathcal{I}^1$ and $\mathcal{I}^2$. The semiflow $\phi$ associated to $\mathcal{I}$ can be constructed by using only elements in $\mathcal{I}^1$, therefore it coincides with $\phi^1$. By the same argument, it also coincides with $\phi^2$, thus $\phi^1 = \phi^2$.
 
 Proving that $\phi$ is well defined as a flow follows exactly in the same way.
\end{proof}

In the next lemma, we collect some properties of the flow $\psi$ and the inverse of its derivative.

\begin{lemma}\label{lemma:prop_flows}
 Let $\mathbf{x}$ be a weak geometric $p$-rough path on $[0,T]$ and let $\omega$ be a control function which controls its $p$-variation. Let $\sigma \in \operatorname{Lip}^{\gamma}$ for some $\gamma > p$ and choose $\nu \geq |\sigma|_{\operatorname{Lip}^{\gamma}}$. Consider the equation
 \begin{align}\label{eqn:rough_DE_flow_2}
  d y_t = \sigma(y_t)\, d \mathbf{x}_t; \qquad t \in [0,T].
 \end{align}
 \begin{itemize}
  \item[(i)] The solution flow $\psi$ to \eqref{eqn:rough_DE_flow_2} exists and there is a constant $C = C(\gamma,p)$ such that
  \begin{align*}
   |\psi(s,v,\xi) - \psi(s,u,\xi)| \leq C( \nu \omega(u,v)^{1/p} \vee \nu^p \omega(u,v)) 
  \end{align*}
  and
  \begin{align*}
   |\psi(s,v,\xi) - \psi(s,u,\xi) - \psi(s,v,\zeta) + \psi(s,u,\zeta)| \leq C \nu \omega(u,v)^{1/p}|\xi - \zeta| \exp(C \nu^p \omega(s,t)) 
  \end{align*}
  holds for every $s \leq u \leq v \leq t$, $s,t \in [0,T]$ and $\xi, \zeta \in \R^m$.
  
  \item[(ii)] For every $s,t \in [0,T]$, $J(s,t,\xi) := (D_{\xi} \psi(s,t,\xi))^{-1}$ exists and satisfies the bound
  \begin{align*}
   |J(s,t,\xi) - I_m| \leq C \nu \omega(s,t)^{1/p} \exp(C \nu^p \omega(s,t))
  \end{align*}
  for every $s \leq t$, $s,t \in [0,T]$ and $\xi \in \R^m$ where $I_m$ denotes the identity matrix in $\R^{m \times m}$ and $C$ as in (i).

 \end{itemize}

\end{lemma}

\begin{proof}
 Claim (i) is a slight generalization of \cite[Theorem 10.14 and Theorem 10.26]{FV10} when we start the equation at time point $s$ instead of $0$. 
 
 Concerning claim (ii), note first that the derivative $D_{\xi} \psi(s,t,\xi)$ and the derivative of the inverse map $D_{\xi} \psi^{-1}(s,t,\xi)$ exist by \cite[Proposition 11.11]{FV10}. Fix $s$ and $t$. From (i), note that for every $h > 0$ we have
  \begin{align*}
   |\psi(s,t,\xi + h e_i) - \psi(s,t,\xi) - h e_i  | \leq h C \nu \omega(s,t)^{1/p} \exp(C \nu^p \omega(s,t)).
  \end{align*}
  Dividing the equation by $h$ and sending $h \to 0$ shows that
  \begin{align*}
    | \partial_{\xi_i} \psi(s,t,\xi) - e_i  | \leq C \nu \omega(s,t)^{1/p} \exp(C \nu^p \omega(s,t))
  \end{align*}
  for every $i = 1,\ldots, m$, thus
  \begin{align*}
   | D_{\xi} \psi(s,t,\xi) - I_m  | \leq C \nu \omega(s,t)^{1/p} \exp(C \nu^p \omega(s,t)).
  \end{align*}
  The inverse flow $\psi^{-1}$ is given by solving \eqref{eqn:rough_DE_flow_2} where we let the rough path $\mathbf{x}$ run ``backwards in time'' from $t$ to $s$ (cf. \cite[Section 11.2]{FV10}). Note that the $p$-variation is invariant under time reversion. Therefore, the same estimate holds for $\psi$ replaced by $\psi^{-1}$.
  From the chain rule,
  \begin{align*}
   J(s,t,\xi) = (D_{\xi}\psi(s,t,\xi))^{-1} = D_{\zeta} \psi^{-1}(s,t,\zeta)\vert_{\zeta = \psi(s,t,\xi)}
  \end{align*}
  from which we can deduce claim (ii).

\end{proof}

\section{RDE flows for a drift with linear growth}\label{sec:RDE_flows}

The next theorem is our first main result.

\begin{theorem}\label{thm:flow_lin_growth}
 Let $\mathbf{x}$ be a weak geometric $p$-rough path with values in $G^{\lfloor p \rfloor}(\R^d)$, $p \in [1,\infty)$, and let $\omega$ be a control function which controls its $p$-variation. Assume that $\sigma = (\sigma_1,\ldots,\sigma_d)$ is a collection of $\operatorname{Lip}^{\gamma + 1}$-vector fields on $\R^m$ for some $\gamma > p$ and choose $\nu \geq |\sigma|_{\operatorname{Lip}^{\gamma}}$. Moreover, assume that $b$ is a locally Lipschitz continuous vector field with linear growth on $\R^m$, i.e. there are some constants $\kappa_1, \kappa_2 \geq 0$ such that
 \begin{align*}
  |b(\xi)| \leq \kappa_1 + \kappa_2|\xi| \quad \text{for all } \xi \in \R^m.
 \end{align*}
 Then the following holds true:
 \begin{enumerate}
  \item The solution flow $\phi$ to  \eqref{eqn:rde_with_drift} exists and is continuous.
  \item There is a constant $C$ depending on $p$, $\gamma$, $\kappa_1$, $\kappa_2$ and $\nu$ such that
  \begin{align}\label{eqn:sup_norm_bound}
   \| \phi(0,\cdot,\xi)\|_{\infty;[0,T]} \leq C \exp(2 \kappa_2 T)(1 + N_1(\omega) + |\xi| + T)
  \end{align}
  and
  \begin{align}\label{eqn:p-var_bound}
    \begin{split}
   \| \phi(0,\cdot,\xi)  \|_{p-\text{var};[0,T]} \leq C \exp(C \kappa_2 T)(1 + N_1(\omega) + \kappa_2 |\xi| + T)
    \end{split}
  \end{align}
  hold for every $\xi \in \R^m$. If $\mathbf{x}$ is a $1/p$-H\"older rough path,
  \begin{align}\label{eqn:1/p_hoelder_bound}
    \sup_{0 \leq s < t \leq T} \frac{|\phi(0,t,\xi) - \phi(0,s,\xi)|}{|t-s| \vee |t - s|^{1/p}} \leq C \exp(C \kappa_2 T)(1 + \kappa_2 |\xi| +  \| \mathbf{x} \|_{1/p-\text{H\"ol};[0,T]} \vee \| \mathbf{x} \|_{1/p-\text{H\"ol};[0,T]}^p).
  \end{align}

  \item For every initial condition $\xi \in \R^m$, the path $t \mapsto \phi(0,t,\xi)$ is a solution to \eqref{eqn:rde_with_drift} in the sense of Friz-Victoir.
 \end{enumerate}

\end{theorem}

\begin{remark}
 In case $b$ is bounded, we can choose $\kappa_2 = 0$. In this case, the estimates we obtain in Theorem \ref{thm:flow_lin_growth} are (essentially) sharp and (basically) coincide with those derived in \cite[Theorem 10.14]{FV10}. 
\end{remark}

\begin{proof}[Proof of Theorem \ref{thm:flow_lin_growth}]
In the following proof, $C$ will be a constant which may depend on $p$, $\gamma$, $\kappa_1$, $\kappa_2$ and $\nu$ but whose actual value may change from line to line. 

From our assumptions on $\sigma$, we know from \cite[Proposition 11.11]{FV10} that the solution flow $\psi$ to \eqref{eqn:rough_DE_flow} exists, is twice differentiable and has a twice differentiable inverse. For $s,t \in [0,T]$ and $\xi \in \R^m$, set 
 \begin{align*}
  J(s,t,\xi) := (D_{\xi} \psi(s,t,\xi))^{-1} \in \R^{m \times m}.
 \end{align*}
  The same proposition states that the maps $(t,\xi)  \mapsto D_{\xi} \psi(0,t,\xi)$ and $(t,\xi) \mapsto D_{\xi}^2 \psi(0,t,\xi)$ are bounded, and the same is true for the inverse $\psi^{-1}$. By the flow property, also the two time parameter flow maps $(s,t,\xi)  \mapsto D_{\xi} \psi(s,t,\xi)$ and $(s,t,\xi) \mapsto D_{\xi}^2 \psi(s,t,\xi)$ are bounded, and the same holds for $\psi^{-1}$. From the chain rule, 
  $(s,t,\xi)  \mapsto J(s,t,\xi)$ and $(s,t,\xi)  \mapsto D_{\xi} J(s,t,\xi)$ are also bounded. Hence the map $(s,t,\xi) \mapsto J(s,t,\xi) b(\psi(s,t,\xi))$ is continuous, locally Lipschitz continuous in space and grows at most linearly in space. Thus the ordinary differential equation \eqref{eqn:ODE_flow} has unique solutions on every time interval, forward and backward in time. Therefore $\chi$ (defined as in Definition \ref{def:sol_flow_rde_with_drift}) exists and the flow $\phi$ is well defined by Lemma \ref{lemma:flow_well_defined}.
  
  We proceed by proving the bound \eqref{eqn:sup_norm_bound} for the sup-norm of $\phi$.
  From Lemma \ref{lemma:prop_flows}, we know that there is a constant $C = C(p,\gamma)$ such that
  \begin{align*}
   \sup_{\xi \in \R^m} |\psi(s,t,\xi) - \xi| \leq C(\nu \omega(s,t)^{1/p} \vee \nu^p \omega(s,t))
  \end{align*}
  and
  \begin{align*}
   \sup_{\xi \in \R^m} |J(s,t,\xi) - I_m| \leq C \nu \omega(s,t)^{1/p} \exp(C \nu^p \omega(s,t))
  \end{align*}
  hold for every $s < t$. Choose $\delta \in (0,1]$ small enough such that
  \begin{align}\label{eqn:delta_cond}
   2(\kappa_1 + \kappa_2)\delta \vee [C \nu \delta^{1/p} \exp(C \nu^p \delta)] \vee [C (\nu \delta^{1/p} \vee \nu^p \delta)] \leq 1
  \end{align}
  and set $\tilde{\omega}(s,t) := \omega(s,t) + |t-s|$. 
  Define a sequence $(\tau_n)$ as follows: set $\tau_0 := 0$ and
  \begin{align*}
   \tau_{n+1} := \inf\{ u\, :\, \tilde{\omega} (\tau_n,u) \geq \delta,\, \tau_n < u \leq T \} \wedge T.
  \end{align*}
  Moreover, set
  \begin{align*}
   N:= N_\delta(\tilde{\omega}) = \sup \{ n \in \N\,:\, \tau_n < T \}.
 \end{align*}
  It follows that $|\tau_{n+1} - \tau_n| \leq \delta$ and $\omega(\tau_n,\tau_{n+1}) \leq \delta$ for every $n = 0,\ldots,N$. Lemma \ref{lemma:prop_flows} implies that for every $t \in [\tau_n,\tau_{n+1}]$, $n = 0,\ldots,N$,
  \begin{align*}
   \sup_{\xi} |\psi(\tau_n,t,\xi) - \xi| \leq 1
  \end{align*}
  and by the triangle inequality,
  \begin{align*}
   \sup_{\xi} |J(\tau_n,t,\xi)| \leq 2.
  \end{align*}
  Now let $y \colon [0,\tau_1] \to \R^m$ be the solution to 
  \begin{align}\label{eqn:ODE_once_again}
   \dot{y}_t = J(0,t,y_t) b(\psi(0,t,y_t)); \quad t \in [0,\tau_1]
  \end{align}
  with initial condition $y_0 = \xi$. Integrating the equation, we obtain for every $t \in [0,\tau_1]$
  \begin{align*}
   |y_t| &\leq |\xi| + 2 \kappa_1 t + 2 \kappa_2 \int_0^t |\psi(0,s,y_s)|\, ds \\
   &\leq |\xi| + 2(\kappa_1 + \kappa_2) t + 2 \kappa_2 \int_0^t |y_s| \, ds.
  \end{align*}
  Gronwall's Lemma implies that
  \begin{align*}
    |\chi_0(t,\xi)| = |y_t| \leq \exp(2 \kappa_2 t)(|\xi| + 2(\kappa_1 + \kappa_2) t) \leq \exp(2 \kappa_2 \tau_1)(|\xi| + 1)
  \end{align*}
  for every $t \in [0,\tau_1]$. Repeating the same argument shows that
  \begin{align*}
   |\chi_{\tau_n}(t,\xi)| \leq \exp(2 \kappa_2 |\tau_{n+1} - \tau_n|)(|\xi| + 1)
  \end{align*}
  holds for every $t \in [\tau_n,\tau_{n+1}]$. Fix some $n \in \{0, \ldots, N\}$ and some $t \in [\tau_n, \tau_{n+1}]$. From the flow property of $\phi$,
  \begin{align*}
   |\phi(0,t,\xi)| &= |\phi(\tau_n,t,\phi(0,\tau_n,\xi))| = |\psi(\tau_n,t,\chi_{\tau_n}(t,\phi(0,\tau_n,\xi)))| \\
   &\leq |\psi(\tau_n,t,\chi_{\tau_n}(t,\phi(0,\tau_n,\xi))) - \chi_{\tau_n}(t,\phi(0,\tau_n,\xi))| + |\chi_{\tau_n}(t,\phi(0,\tau_n,\xi))| \\
   &\leq 1 + \exp(2 \kappa_2 |\tau_{n+1} - \tau_n|) + \exp(2 \kappa_2 |\tau_{n+1} - \tau_n|)|\phi(0,\tau_n,\xi)|.
  \end{align*}
  Set $\phi_n := \sup_{t \in [\tau_n,\tau_{n+1}]} |\phi(0,t,\xi)|$ and $C_n := \exp(2 \kappa_2 |\tau_{n+1} - \tau_n|)$. The estimate above reads
  \begin{align*}
   \phi_0 \leq 1 + C_0(1 + |\xi|) \qquad \text{and} \qquad   \phi_{n+1} \leq 1 + C_{n+1} + C_{n+1} \phi_n 
  \end{align*}
  for every $n = 0,\ldots,N$. By induction,
  \begin{align*}
   \phi_n &\leq 1 + 2(C_n + C_n C_{n-1} + \ldots + C_n \cdots C_1) + C_n \cdots C_1 C_0 (1 + |\xi|) \\
   &\leq 1 + C_0 \ldots C_N(1 + 2N + |\xi|).
  \end{align*}
  This implies that
  \begin{align*}
   \sup_{t \in [0,T]} |\phi(0,t,\xi)| \leq 1 + \exp(2\kappa_2 T)(1 + 2N + |\xi|).
  \end{align*}
  Next,
  \begin{align*}
   N_{\delta}(\tilde{\omega}) \leq 2N_{\delta}(\omega) + 2T/\delta + 2 \leq 4 N_1(\omega) /\delta  + 2/\delta + 2T/\delta + 2
  \end{align*}
  where the first inequality follows from \cite[Lemma 5]{BFRS16}, the second one uses \cite[Lemma 3]{FR12b}. This implies the bound \eqref{eqn:sup_norm_bound}.

  We proceed with proving \eqref{eqn:p-var_bound}. Choose $\delta \in (0,1]$ as in \eqref{eqn:delta_cond} with the additional condition $\delta \leq \nu^{-p}$. We define $\tilde{\omega}$, $(\tau_n)$ and $N = N_{\delta}(\tilde{\omega})$ as above. If $y \colon [0,\tau_1] \to \R^m$ denotes the solution to \eqref{eqn:ODE_once_again} with initial condition $y_0 = \xi$, we have for $u < v$, $u,v \in [0,\tau_1]$,
  \begin{align*}
   |y_v - y_u| &\leq 2(\kappa_1 + \kappa_2) |v - u| + 2 \kappa_2 \sup_{t \in [0,\tau_1]} |y_t| |v - u| + 2\kappa_2 \int_u^v |y_s - y_u|\, ds \\
   &\leq 2(\kappa_1 + \kappa_2) |v - u| + 2\kappa_2 \exp(2 \kappa_2 \tau_1)(1 + |\xi|) |v - u| + 2 \kappa_2 \int_u^v |y_s - y_u|\, ds. \\
  \end{align*}
  From Gronwall's Lemma,
  \begin{align*}
   |\chi_0(v,\xi) - \chi_0(u,\xi)| &= |y_v - y_u| \leq |v - u| \left[ 2( \kappa_1 + \kappa_2) + 2 \kappa_2 \exp(2 \kappa_2 \tau_1)(1 + |\xi|) \right] \exp(2 \kappa_2 \tau_1) \\
   &\leq 2 \exp(2) [ (\kappa_1 + \kappa_2) + \kappa_2 (1 + |\xi|)]   |v - u|.
  \end{align*}
  Similarly, one can show that for every $n = 0,\ldots, N$, $u,v \in [\tau_n, \tau_{n+1}]$ and $u<v$,
   \begin{align}\label{eqn:chi_estim}
   |\chi_{\tau_n}(v,\xi) - \chi_{\tau_n}(u,\xi)| \leq 2 \exp(2) [ (\kappa_1 + \kappa_2) + \kappa_2 (1 + |\xi|)]   |v - u|.
  \end{align}
  Fix $n \in \{0,\ldots,N\}$ and $u,v \in [\tau_n, \tau_{n+1}]$ with $u<v$. Using the flow property,
  \begin{align*}
   |\phi(0,u,\xi) - \phi(0,v,\xi)| &= | \psi(\tau_n,u,\chi_{\tau_n}(u,\phi(0,\tau_n,\xi))) - \psi(\tau_n,v,\chi_{\tau_n}(v,\phi(0,\tau_n,\xi)))| \\
   &\leq | \psi(\tau_n,u,\chi_{\tau_n}(u,\phi(0,\tau_n,\xi))) - \psi(\tau_n,v,\chi_{\tau_n}(u,\phi(0,\tau_n,\xi))) | \\
   &\quad + |\psi(\tau_n,v,\chi_{\tau_n}(u,\phi(0,\tau_n,\xi))) - \psi(\tau_n,v,\chi_{\tau_n}(v,\phi(0,\tau_n,\xi)))|.
  \end{align*}
  For the first term, we use Lemma \ref{lemma:prop_flows} to see that
  \begin{align*}
   | \psi(\tau_n,u,\chi_{\tau_n}(u,\phi(0,\tau_n,\xi))) - \psi(\tau_n,v,\chi_{\tau_n}(u,\phi(0,\tau_n,\xi))) | \leq C \nu \omega(u,v)^{1/p}.
  \end{align*}
  For the second term, we use again Lemma \ref{lemma:prop_flows}, the triangle inequality and the estimate \eqref{eqn:chi_estim} to see that
  \begin{align*}
    &|\psi(\tau_n,v,\chi_{\tau_n}(u,\phi(0,\tau_n,\xi))) - \psi(\tau_n,v,\chi_{\tau_n}(v,\phi(0,\tau_n,\xi)))| \\
    \leq\ &|\psi(\tau_n,v,\chi_{\tau_n}(u,\phi(0,\tau_n,\xi))) - \psi(\tau_n,v,\chi_{\tau_n}(v,\phi(0,\tau_n,\xi))) - \chi_{\tau_n}(u,\phi(0,\tau_n,\xi)) + \chi_{\tau_n}(v,\phi(0,\tau_n,\xi))|\\
    &\quad + |\chi_{\tau_n}(u,\phi(0,\tau_n,\xi)) - \chi_{\tau_n}(v,\phi(0,\tau_n,\xi))| \\
    \leq\ &2|\chi_{\tau_n}(u,\phi(0,\tau_n,\xi)) - \chi_{\tau_n}(v,\phi(0,\tau_n,\xi))| \\
    \leq\ &4 \exp(2) [ (\kappa_1 + \kappa_2) + \kappa_2 (1 + |\phi(0,\tau_n,\xi)|)]   |v - u| 
  \end{align*}
Putting these estimates together, we have shown that for any $u,v \in [\tau_n, \tau_{n+1}]$,
  \begin{align}\label{eqn:lip_prop}
   |\phi(0,u,\xi) - \phi(0,v,\xi)| \leq C \left[1 + \kappa_2 (1 + |\phi(0,\tau_n,\xi)|) \right] |v- u| + C \omega(u,v)^{1/p} 
  \end{align}
  which implies that
  \begin{align*}
   \| \phi(0,\cdot,\xi) \|_{p-\text{var};[\tau_n,\tau_{n+1}]} \leq  C \left[1 + \kappa_2 (1 + |\phi(0,\tau_n,\xi)|) \right] |\tau_{n+1} - \tau_n| + C \delta^{1/p}.
  \end{align*}
  Now
  \begin{align*}
   \| \phi(0,\cdot,\xi) \|_{p-\text{var};[0,T]} &\leq \sum_{n = 0}^N \| \phi(0,\cdot,\xi) \|_{p-\text{var};[\tau_n,\tau_{n+1}]} \\
   &\leq C \left[1 + \kappa_2 (1 + \| \phi(0,\cdot,\xi) \|_{\infty;[0,T]}) \right] T + C (N_{\delta}(\tilde{\omega}) + 1) \delta^{1/p}
  \end{align*}
  The claim follows by using the bound \eqref{eqn:sup_norm_bound} for the sup-norm and again \cite[Lemma 5]{BFRS16} and \cite[Lemma 3]{FR12b}.
  
  Next we prove the bound \eqref{eqn:1/p_hoelder_bound} in the H\"older case. Here, we may choose $\omega(s,t) = \|\mathbf{x}\|_{1/p - \text{H\"ol}}^p |t - s|$ which implies $N_1(\omega) \leq  \|\mathbf{x}\|_{1/p - \text{H\"ol}}^p T$. As for \eqref{eqn:lip_prop}, we can show that for every $u \leq v$ for which
  \begin{align*}
    \|\mathbf{x}\|_{1/p - \text{H\"ol}}^p |v - u| = \omega(u,v) \leq \delta
  \end{align*}
  holds, we have
   \begin{align*}
   |\phi(0,u,\xi) - \phi(0,v,\xi)| \leq C \left[1 + \kappa_2 (1 + \|\phi(0,\cdot,\xi)\|_{\infty;[0,T]}) \right] |v- u| + C \|\mathbf{x}\|_{1/p - \text{H\"ol}} |v - u|^{1/p}.
  \end{align*}
  We claim that for every $u \leq v$, $u,v \in [0,T]$, we have
   \begin{align*}
   |\phi(0,u,\xi) - \phi(0,v,\xi)| &\leq C \left[1 + \kappa_2 (1 + \|\phi(0,\cdot,\xi)\|_{\infty;[0,T]}) \right] |v - u| \\
   &\quad + C \left[ ( \|\mathbf{x}\|_{1/p - \text{H\"ol}}|v - u|^{1/p}) \vee (\|\mathbf{x}\|_{1/p - \text{H\"ol}}^p |v - u|) \right].
  \end{align*}
  We prove this similarly to \cite[Exercise 4.24]{FH14}. First, there is nothing to show for $v - u \leq h := \delta \|\mathbf{x}\|_{1/p - \text{H\"ol}}^{-p}$. If this is not the case, we define $t_i := (u + ih) \wedge v$ and observe that $t_M = v$ for $M \geq (v - u)/h$ and that $t_{i+1} - t_i \leq h$. Then
  \begin{align*}
   |\phi(0,u,\xi) - \phi(0,v,\xi)| &\leq \sum_{0 \leq i < (v - u)/h} |\phi(0,t_{i+1},\xi) - \phi(0,t_i,\xi)| \\
   &\leq C \left[1 + \kappa_2 (1 + \|\phi(0,\cdot,\xi)\|_{\infty;[0,T]}) \right] |v- u| + C \|\mathbf{x}\|_{1/p - \text{H\"ol}} h^{1/p}(1 + |v - u|/h) \\
   &\leq C \left[1 + \kappa_2 (1 + \|\phi(0,\cdot,\xi)\|_{\infty;[0,T]}) \right] |v- u| + 2C \|\mathbf{x}\|_{1/p - \text{H\"ol}} h^{1/p - 1}|v - u|
  \end{align*}
  which shows the claim by definition of $h$. This also implies \eqref{eqn:1/p_hoelder_bound} by using the sup-bound of our solution and the bound for $N_1(\omega)$ we saw above.

  Next we show that $\phi$ is continuous. From the flow property, for $s \leq u \leq v \leq t$,
  \begin{align*}
   \phi(s,t,\xi) = \phi(v,t,\phi(u,v,\phi(s,u,\xi))),
  \end{align*}
  therefore it suffices to prove that for some given $\delta > 0$, $[s,t] \times \R^m \ni (u,\xi) \mapsto \phi(u,v_0,\xi)$ and $[s,t] \times \R^m \ni (v,\xi) \mapsto \phi(u_0,v,\xi)$ are continuous for fixed $v_0$ resp. $u_0$ where $u_0 ,v_0 \in [s,t]$ and $s < t$ satisfy $|t - s| \leq \delta$. By continuity of $\psi$, it suffices to show that $(u,\xi) \mapsto \chi_{u}(v_0,\xi)$ and $(v,\xi) \mapsto \chi_{u_0}(v,\xi)$ are continuous. This, however, follows by standard arguments for ordinary differential equations (or see the proof of the forthcoming Theorem \ref{thm:main_semiflow} where this is carried out in more detail in even more generality).

  It remains to show that for fixed $\xi \in \R^m$, the path $y_t := \phi^{\mathbf{x}}(0,t,\xi)$ is a solution to \eqref{eqn:rde_with_drift} in the sense of Friz-Victoir. We will give the proof in even more generality in the forthcoming Theorem \ref{thm:main_semiflow}.

\end{proof}

\section{RDE semiflows}\label{sec:RDE_semiflows}

Next, our goal is to further relax the assumptions on $b$ which will yield a \emph{semi}flow of \eqref{eqn:rde_with_drift}. We first prove an a priori estimate for ordinary differential equations.

\begin{lemma}\label{lemma:a_priori_est_ODE}
 Consider an ordinary differential equation of the form
 \begin{align}\label{eqn:ordinary_diff_eq}
   \dot{z}_t = J(t,z_t)b(\psi(t,z_t)); \qquad t \in [0,T]
 \end{align}
 where
 \begin{enumerate}
     \item $b \colon \R^m \to \R^m$ is continuous and satisfies the following conditions:
    \begin{itemize}
     \item[(i)] There exists a constant $C_1$ such that
     \begin{align}\label{eqn:bound_radial_growth}
      \langle b(\xi), \xi \rangle \leq C_1(1 + |\xi|^2) \quad \text{for every } \xi \in \R^m.
     \end{align}
     \item[(ii)] There exists a constant $C_2$ such that
     \begin{align}\label{eqn:bound_tang_growth}
      \left| b(\xi) - \frac{\langle b(\xi),\xi \rangle \xi}{| \xi |^2} \right| \leq C_2(1 + |\xi|) \quad \text{for every } \xi \in \R^m \setminus \{0\}.
     \end{align}

    \end{itemize}

    \item $J \colon [0,T] \times \R^m \to \R^{m \times m}$ is continuous and satisfies 
    \begin{align*}
     \sup_{t \in [0,T];\ \xi \in \R^m} |J(t,\xi) - I_m | \leq \frac{1}{2}
    \end{align*}
    where $I_m$ denotes the identity matrix in $\R^{m \times m}$.

    \item $\psi \colon [0,T] \times \R^m \to \R^m$ is continuous
    and there exists a constant $C_3$ such that
    \begin{align*}
     \sup_{t \in [0,T];\ \xi \in \R^m} | \psi(t,\xi) - \xi | \leq C_3.
    \end{align*}

 \end{enumerate}
 
 Then any solution $z \colon [0,T] \to \R^m$ to \eqref{eqn:ordinary_diff_eq} with initial condition $z_0 = \xi \in \R^m$ satisfies the bounds
  \begin{align*}
   \|z \|_{\infty;[0,T]} \leq (CT + |\xi|) e^{C T} 
  \end{align*}
  and
  \begin{align*}
  \|z \|_{1-\text{var};[0,T]} &\leq C(1 + \|z\|_{\infty;[0,T]})T + (|\xi| - \hat{C})^+ - (|z_T| - \hat{C})^+ \\
  \end{align*}
  where 
  \begin{align*}
   \hat{C} = (C_3 + 1) \vee (4C_3)
  \end{align*}
  and $C$ is a constant depending on $C_i$, $i = 1,2,3,4$ where 
 \begin{align*}
  C_4 := \sup \left\{ |b(\xi)| \, :\, |\xi| \leq (2 C_3 + 1)\vee(5 C_3) + 1 \right\}.
 \end{align*}

\end{lemma}

\begin{proof}
 Let $z \colon [0, T] \to \R^m$ be a solution to \eqref{eqn:ordinary_diff_eq} with initial condition $z_0 = \xi \in \R^m$. For $t \in [0,T]$, set $h_t := \psi(t,z_t) - z_t$. Note that by assumption, $|h_t| \leq C_3$ for every $t$. From the chain rule,
 \begin{align*}
  \frac{d}{dt} |z_t|^2 = 2 \langle z_t,\dot{z}_t \rangle = 2 \langle z_t, J(t,z_t)b(z_t + h_t) \rangle.
 \end{align*}
 Fix $t \in [0,T]$. To simplify notation, set $z = z_t$ and $h = h_t$. We aim to show that there exists a constant $C$ depending on $C_i$, $i = 1,2,3,4$, but independent of $t$ such that
 \begin{align}\label{eqn:claim}
  \langle z, J(t,z)b(z + h) \rangle \leq C(1 + |z|^2).
 \end{align}
  Note that the bound clearly holds for $|z| \leq (C_3 + 1) \vee (4 C_3)$ since $J$ is bounded and $|z + h| \leq (2 C_3 + 1)\vee(5 C_3)$ in this case. From now on, we assume that $|z| \geq (C_3 + 1) \vee (4 C_3)$. Let $b(z + h) = \alpha(z + h) + \beta v$ where $\alpha, \beta \in \R$ and $v \perp (z + h)$, $|v| = 1$. From \eqref{eqn:bound_radial_growth}, we see that
  \begin{align*}
   C_1(1 + |z + h|^2) \geq \langle b(z+h),z+h \rangle = \alpha |z + h|^2.
  \end{align*}
  Since $|z| \geq C_3 + 1$, we have $|z + h| \geq 1$ which implies that $\alpha \leq 2 C_1$. The bound \eqref{eqn:bound_tang_growth} implies that $|\beta| \leq C_2(1 + |z + h|)$. We have
  \begin{align*}
   \langle z, J(t,z)b(z + h) \rangle = \alpha \langle z, J(t,z)(z + h) \rangle + \beta \langle z, J(t,z)v \rangle.
  \end{align*}
  For the second term, we use the Cauchy Schwarz inequality to see that
  \begin{align*}
   |\beta \langle z, J(t,z)v \rangle| \leq | \beta | |z| |J(t,z)| \leq 3 C_2  |z|^2.
  \end{align*}
  Concerning the first term, note that $|h| \leq C_3 \leq |z|/4$, thus
  \begin{align}\label{eqn:lower_bdd_a_priori}
    \begin{split}
    \langle z, J(t,z)(z + h) \rangle &= \langle z, z + h \rangle + \langle z, (J(t,z) - \operatorname{Id}) (z + h) \rangle\geq |z|^2 - |z||h| - |z|^2/2 - |z||h| / 2 \\
    &\geq |z|^2 - |z|^2/4 - |z|^2/2 - |z|^2/8 = |z|^2/8  > 0. 
    \end{split}
  \end{align}
  Therefore, we obtain the bound
  \begin{align*}
    \alpha \langle z, J(t,z)(z + h) \rangle \leq 4 C_1 |z|^2.
  \end{align*}
  This shows that indeed \eqref{eqn:claim} holds for every $z$. Gronwall's Lemma implies the claim for the sup-norm.
  
  We proceed with the bound for the total variation norm of $t \mapsto z_t$. Let $[a,b]$ be a subinterval of $[0,T]$ on which $|z_t| \leq (C_3 + 1) \vee (4C_3) + 1$ for all $t \in [a,b]$. In this case, for every $a \leq u \leq v \leq b$,
  \begin{align*}
   |z_v - z_u| \leq \int_u^v J(s,z_s)b(z_s + h_s)\, ds \leq \frac{3}{2} C_4 (v - u)
  \end{align*}
  which implies that
  \begin{align}\label{eqn:1varbound_zsmall}
   \| z \|_{1-\text{var};[a,b]} \leq \frac{3}{2} C_4 (b - a).
  \end{align}
  Now assume that $[a,b]$ is a subinterval of $[0,T]$ on which $|z_t| \geq (C_3 + 1) \vee (4C_3)$ for all $t \in [a,b]$. In a first step, we show that the total variation of $t \mapsto |z_t|$ has a good bound on $[a,b]$. As above, we can show that
  \begin{align}\label{eqn:deriv_bound_pvar}
   \frac{d}{dt} |z_t| =  \frac{\langle z_t, \dot{z}_t \rangle}{|z_t|} \leq C(1 + |z_t|) \leq C(1 + \|z\|_{\infty;[a,b]})
  \end{align}
  for $t \in (a,b)$. Since $t \mapsto |z_t| =: f(t)$ has finite total variation, there is a decomposition $f = f^+ - f^-$ where $f^+$ and $f^-$ are increasing functions, and
  \begin{align*}
   \|f\|_{1-\text{var};[a,b]} = \|f^+\|_{1-\text{var};[a,b]} + \|f^-\|_{1-\text{var};[a,b]}.
  \end{align*}
  The estimate \eqref{eqn:deriv_bound_pvar} implies that
  \begin{align*}
    \|f^+\|_{1-\text{var};[a,b]} \leq C(1 + \|z\|_{\infty;[a,b]})(b - a).
  \end{align*}
  Since $f$ is nonnegative, $\|f^-\|_{1-\text{var};[a, b]} = \|f^+\|_{1-\text{var};[a, b]} + f(a) - f(b) $, therefore
  \begin{align}\label{eqn:tv_bound_abs_value}
   \| |z| \|_{1-\text{var};[a,b]} = \int_{a}^{b} \left| \frac{d}{dt} |z_t| \right|\, dt \leq 2C(1 + \|z\|_{\infty;[a,b]})(b - a) + |z_a| - |z_b|.
  \end{align}
  We proceed with proving a bound for the total variation of $t \mapsto z_t$ on $[a,b]$. By the triangle inequality,
  \begin{align}\label{eqn:triangle_deriv}
   |\dot{z}_t| \leq \left| \langle \dot{z}_t, \frac{z_t}{|z_t|} \rangle \right| + \left| \dot{z}_t - \langle \dot{z}_t, \frac{z_t}{|z_t|} \rangle \frac{z_t}{|z_t|} \right|.
  \end{align}
  We will first estimate the second term. Fix $t \in (a,b)$ and define $h$ as above. As before, we decompose $b(z + h) = \alpha(z + h) + \beta v$ with $v \perp (z + h)$, $|v| = 1$. Then,
  \begin{align*}
   \left| \dot{z}_t - \langle \dot{z}_t, \frac{z_t}{|z_t|} \rangle \frac{z_t}{|z_t|} \right| &= \left| J(t,z)b(z + h) - \langle J(t,z)b(z + h), \frac{z}{|z|} \rangle \frac{z}{|z|} \right| \\
   &= \left| \alpha J(t,z)(z + h) + \beta J(t,z) v  - \langle J(t,z)(\alpha(z + h) + \beta v), \frac{z}{|z|} \rangle \frac{z}{|z|} \right| \\
    &\leq |\alpha| \left| J(t,z)(z + h) - \langle J(t,z)(z + h) , \frac{z}{|z|} \rangle \frac{z}{|z|} \right| + 3 |\beta|.
  \end{align*}
  Note that the vector $\langle J(t,z)(z + h) , \frac{z}{|z|} \rangle \frac{z}{|z|}$ is the orthogonal projection of $J(t,z)(z + h)$ on the space $\text{span}\{z\}$. Since the distance between $J(t,z)(z + h)$ and its orthogonal projection is minimal among all elements in $\text{span}\{z\}$, we obtain in particular
  \begin{align*}
    \left| J(t,z)(z + h) - \langle J(t,z)(z + h) , \frac{z}{|z|} \rangle \frac{z}{|z|} \right| &\leq  \left| J(t,z)(z + h) - z \right| \\
    &= |(J(t,z) - I_m)z + J(t,z)h | \\
    &\leq \frac{1}{2}|z| + \frac{3}{2}|h| \\
    &\leq \frac{7}{8}|z|.
  \end{align*}
  In \eqref{eqn:lower_bdd_a_priori} we have seen that
  \begin{align*}
   \left| \langle J(t,z)(z+h), z/|z| \rangle \right| \geq  \frac{1}{8}|z|.
  \end{align*}
  Putting these estimates together implies that
  \begin{align*}
   \left| \dot{z}_t - \langle \dot{z}_t, \frac{z_t}{|z_t|} \rangle \frac{z_t}{|z_t|} \right| &\leq 3 |\beta| + 7|\alpha| \left| \langle J(t,z)(z+h), z/|z| \rangle \right| \\
   &\leq 3 |\beta| + 7 \left| \langle J(t,z)(\alpha (z+h) + \beta v), z/|z| \rangle \right| + 7 |\beta| |J(t,z)| \\
   &\leq \frac{27}{2} |\beta| + 7 \left| \langle \dot{z}_t, \frac{z_t}{|z_t|} \rangle \right|.
  \end{align*}
  Going back to \eqref{eqn:triangle_deriv}, we obtain the estimate
  \begin{align*}
   |\dot{z}_t| \leq \frac{27}{2} |\beta| + 8 \left| \langle \dot{z}_t, \frac{z_t}{|z_t|} \rangle \right|
  \end{align*}
  for every $t \in (a,b)$.  We have already seen that $|\beta| \leq C(1 + |z_t|)$ for some constant $C$. From \eqref{eqn:tv_bound_abs_value},
  \begin{align*}
   \int_{a}^{b} \left| \langle \dot{z}_t, \frac{z_t}{|z_t|} \rangle \right|\, dt \leq 2C(1 + \|z\|_{\infty;[a,b]})(b - a) + |z_a| - |z_b|.
  \end{align*}
  Therefore, we can conclude that there is a constant $C$ such that
  \begin{align}\label{eqn:1varbound_zlarge}
   \| z \|_{1-\text{var};[a,b]} = \int_a^b \left|  \dot{z}_t \right|\, dt \leq C(1 + \|z\|_{\infty;[a,b]})(b - a) + |z_a| - |z_b|.
  \end{align}
  Now define
  \begin{align*}
   S := \{t \in (0,T) :\, |z_t| < \hat{C} + 1 \} \quad \text{and} \quad U := \{t \in (0,T) \, :\, |z_t| > \hat{C} \}.
  \end{align*}
  Note that both sets are open in $\R$. Hence, there are countable sets $I$ and $J$ such that $\{(a_k,b_k) \}_{k \in I}$ and $\{(a_k,b_k) \}_{k \in J}$ are disjoint families of open intervals for which
  \begin{align*}
   S = \bigcup_{k \in I} (a_k,b_k) \quad \text{and} \quad U = \bigcup_{k \in J} (a_k,b_k).
  \end{align*}
  Clearly,
  \begin{align*}
   \| z \|_{1-\text{var};[0,T]} \leq \sum_{k \in I \cup J} \| z \|_{1-\text{var};[a_k,b_k]}.
  \end{align*}
  From \eqref{eqn:1varbound_zsmall}, we have
  \begin{align*}
   \sum_{k \in I} \| z \|_{1-\text{var};[a_k,b_k]} \leq \frac{3}{2} C_4 T.
  \end{align*}
  Note that for any $k \in J$, by continuity, $\lim_{t \searrow a_k} |z_t| > \hat{C}$ implies that $a_k = 0$, therefore $|z_{a_k}| = |\xi|$. By the same reasoning, $\lim_{t \nearrow b_k} |z_t| > \hat{C}$ implies $b_k = T$ and $|z_{b_k}| = |z_T|$. In particular, there are at most two elements $k_1, k_2 \in J$ for which $|z_{a_{k_i}}| \neq |z_{b_{k_i}}|$, $i = 1,2$. Using \eqref{eqn:1varbound_zlarge}, this implies that
  \begin{align*}
   \sum_{k \in J} \| z \|_{1-\text{var};[a_k,b_k]} \leq C(1 + \|z\|_{\infty;[a,b]})T + (|\xi| - \hat{C})^+ - (|z_T| - \hat{C})^+
  \end{align*}
  and we can conclude our assertion.

\end{proof}

The next Lemma states similar conditions which will imply uniqueness.

 \begin{lemma}\label{lemma:uniqueness_ODE}
 Consider two solutions $z^1, z^2 \colon [0,T] \to \R^m$ to the equations
  \begin{align}\label{eqn:2_ordinary_diff_eq}
    \begin{split}
   \dot{z}_t^i &= J^i(t,z_t^i)b(\psi^i(t,z_t^i)); \qquad t \in [0,T] \\
   z^i_0 &= \xi^i \in \R^m; \quad i = 1,2.
    \end{split}
 \end{align}
 Let $R > 0$ be such that $R \geq \sup_{t \in [0,T]} |z^i_t|$ for $i = 1,2$. We assume the following:
 \begin{enumerate}
     \item $b \colon \R^m \to \R^m$ is continuous and satisfies the following conditions:
    \begin{itemize}
     \item[(i)] There exists a constant $C_1 = C_1(R)$ such that
     \begin{align}\label{eqn:bound_radial_growth_uniq}
      \langle b(\xi) - b(\zeta), \xi - \zeta \rangle \leq C_1|\xi - \zeta|^2 \quad \text{for every } \xi, \zeta \in B(0,R).
     \end{align}
     \item[(ii)] There exists a constant $C_2 = C_2(R)$ such that
     \begin{align}\label{eqn:bound_tang_growth_uniq}
      \left| b(\xi) - b(\zeta) - \frac{\langle b(\xi) - b(\zeta) , \xi - \zeta \rangle (\xi - \zeta)}{| \xi - \zeta |^2} \right| \leq C_2|\xi - \zeta| \quad \text{for every } \xi, \zeta \in B(0,R) \text{ with } \xi - \zeta \neq 0.
     \end{align}

    \end{itemize}

    \item Both $J^1, J^2 \colon [0,T] \times \R^m \to \R^{m \times m}$ are continuous and there exists a constant $C_3 = C_3(R)$ such that 
    \begin{align*}
      \sup_{t \in [0,T]} |J^1(t,\xi) - J^1(t,\zeta)| \leq C_3 |\xi - \zeta| \quad \text{for every } \xi, \zeta \in B(0,R).
    \end{align*}
    Moreover, we assume that
    \begin{align*}
     \sup_{t \in [0,T];\ |\xi| \leq R} |J^2(t,\xi) - I_m | \leq \frac{1}{2}
    \end{align*}
    where $I_m$ denotes the identity matrix in $\R^{m \times m}$.

    \item Both $\psi^1, \psi^2 \colon [0,T] \times \R^m \to \R^m$ are continuous and
    \begin{align*}
     \sup_{t \in [0,T]} | \psi^2(t,\xi) - \xi - \psi^2(t,\zeta) + \zeta | \leq \frac{1}{4} |\xi - \zeta| \quad \text{for every } \xi, \zeta\in B(0,R).
    \end{align*}
    
    \item There exists an $\varepsilon > 0$ such that
 \begin{align*}
   &\sup_{t \in [0,T];\ \xi \in B(0,R)} |J^1(t,\xi) - J^2(t,\xi)| \leq \varepsilon \qquad  \text{and} \\
   &\sup_{t \in [0,T];\ \xi \in B(0,R)} |b(\psi^1(t,\xi)) - b(\psi^2(t,\xi))| \leq \varepsilon.
 \end{align*}
 \end{enumerate}

 Let $C_4 \geq \sup_{s \in [0,T],\, |\xi| \leq R} |b(\psi^1(s,\xi))|$. Then there is a constant $C$ depending on $C_i$, $i = 1,2,3,4$, on $R$ and on $T$ such that
 \begin{align*}
  \sup_{t \in [0,T]} |z^1_t - z^2_t| \leq C(|\xi^1 - \xi^2| + \sqrt{\varepsilon}).
 \end{align*}
 
 In particular, if $b$ satisfies the conditions \eqref{eqn:bound_radial_growth_uniq} and \eqref{eqn:bound_tang_growth_uniq} locally on every compact set and if $\psi^1 = \psi^2 =: \psi$ and $J^1 = J^2 =: J$ satisfy the stated spatial conditions globally, solutions to \eqref{eqn:ordinary_diff_eq} with the same initial condition are unique.


\end{lemma}

\begin{proof}
  We have
 \begin{align*}
  |z^1_t - z^2_t|^2 &= |\xi^1 - \xi^2|^2 + 2 \int_0^t \langle z^1_s - z^2_s, J^1(s,z^1_s)b(\psi^1(s,z^1_s)) - J^2(s,z^2_s)b(\psi^2(s,z^2_s)) \rangle\, ds \\
  &= |\xi^1 - \xi^2|^2 + 2\int_0^t \langle z^1_s - z^2_s, (J^1(s,z^1_s) - J^2(s,z^2_s)) b(\psi^1(s,z^1_s))  \rangle\, ds \\
  &\quad + 2\int_0^t \langle z^1_s - z^2_s, J^2(s,z^2_s) (b(\psi^1(s,z^1_s)) - b(\psi^2(s,z^2_s))) \rangle\, ds
 \end{align*}
 for every $t \in [0,T]$. 
 By the Cauchy-Schwarz inequality,
 \begin{align*}
  &\left|\int_0^t \langle z^1_s - z^2_s, (J^1(s,z^1_s) - J^2(s,z^2_s)) b(\psi^1(s,z^1_s))  \rangle\, ds\right| \\
  \leq\ &C_3 C_4 \int_0^t |z^1_s - z^2_s|^2 \, ds + 2 \varepsilon C_4 T R
 \end{align*}
 for every $t \in [0,T]$. We aim to prove that a similar estimate holds for the second integral
 \begin{align}
  &\int_0^t \langle z^1_s - z^2_s, J^2(s,z^2_s) (b(\psi^1(s,z^1_s)) - b(\psi^2(s,z^2_s))) \rangle\, ds \label{eqn:2nd_integral} \\
  =\ &\int_0^t \langle z^1_s - z^2_s, J^2(s,z^2_s) (b(\psi^1(s,z^1_s)) - b(\psi^2(s,z^1_s))) \rangle\, ds \label{eqn:loc_one_sided_lip_1} \\
  &\quad  + \int_0^t \langle z^1_s - z^2_s, J^2(s,z^2_s) (b(\psi^2(s,z^1_s)) - b(\psi^2(s,z^2_s))) \rangle\, ds \label{eqn:loc_one_sided_lip_2}
 \end{align}
 The integral \eqref{eqn:loc_one_sided_lip_1} can be estimated by
 \begin{align*}
  \left| \int_0^t \langle z^1_s - z^2_s, J^2(s,z^2_s) (b(\psi^1(s,z^1_s)) - b(\psi^2(s,z^1_s))) \rangle\, ds \right| \leq 3T R \varepsilon.
 \end{align*}
  We proceed with the integral \eqref{eqn:loc_one_sided_lip_2}. Fix $s \in [0,T]$ and set $h^i_s := \psi^2(s,z^i_s) - z^i_s$. To simplify notation, set $z^i := z^i_s$ and $h^i := h^i_s$, $i = 1,2$. Choose $\alpha$, $\beta \in \R$ such that
 \begin{align*}
  b(z^1 + h^1) - b(z^2 + h^2) = \alpha( z^1 + h^1 - z^2 - h^2) + \beta v
 \end{align*}
 where $v \perp (z^1 + h^1 - z^2 - h^2)$, $|v| = 1$. From our conditions on $b$, 
 \begin{align*}
  \alpha \leq C_1 \qquad \text{and} \qquad |\beta| \leq C_2|z^1 + h^1 - z^2 - h^2|.
 \end{align*}
 Note that
 \begin{align*}
  |h^1 - h^2| = |\psi^2(s,z^1) - z^1 - \psi^2(s,z^2) + z^2| \leq  \frac{1}{4} |z^1 - z^2|
 \end{align*}
 and
 \begin{align*}
  \langle z^1 - z^2, J^2(s,z^2) (b(\psi^2(s,z^1)) - b(\psi^2(s,z^2))) \rangle\ &= \alpha \langle z^1 - z^2, J^2(s,z^2) ( z^1 + h^1 - z^2 - h^2) \rangle \\
  &\quad + \beta \langle z^1 - z^2,J^2(s,z^2) v \rangle.
 \end{align*}
 The second term can be estimated using Cauchy-Schwarz and our bound for $\beta$:
 \begin{align*}
  |\beta \langle z^1 - z^2,J^2(s,z^2) v \rangle| \leq C|z^1 - z^2|^2
 \end{align*}
 where $C$ can be chosen uniformly over $s \in [0,T]$. For the first term, note that
 \begin{align*}
  &\langle z^1 - z^2, J^2(s,z^2) ( z^1 + h^1 - z^2 - h^2) \rangle \\
  =\ &\langle z^1 - z^2,  z^1 + h^1 - z^2 - h^2 \rangle  + \langle z^1 - z^2, (J^2(s,z^2) - \operatorname{Id}) (z^1 + h^1 - z^2 - h^2) \rangle \\
  \geq\ &| z^1 - z^2|^2 - | z^1 - z^2||h^1 - h^2| - \frac{1}{2}| z^1 - z^2|^2 - \frac{1}{2} |z^1 - z^2||h^1 - h^2| \\
  \geq\ &| z^1 - z^2|^2 - \left( \frac{1}{4} + \frac{1}{2} + \frac{1}{8} \right) | z^1 - z^2|^2 \geq 0.
 \end{align*}
 This implies that
 \begin{align*}
  \alpha \langle z^1 - z^2, J^2(s,z^2) ( z^1 + h^1 - z^2 - h^2) \rangle \leq C|z^1 - z^2|^2
 \end{align*}
 for some $C$ which does not depend on $s$. This shows that there is a constant $C$ such that the integral \eqref{eqn:2nd_integral} can be bounded by
 \begin{align*}
  3TR \varepsilon + C\int_0^t |z^1_s - z^2_s|^2\, ds
 \end{align*}
 for every $t \in [0,T]$. Together with our former estimates, this shows that
 \begin{align*}
  |z^1_t - z^2_t|^2 \leq |\xi^1 - \xi^2|^2 + \varepsilon TR(3 + 2 C_4) + C\int_0^t |z^1_s - z^2_s|^2\, ds
 \end{align*}
 for every $t \in [0,T]$. Applying Gronwall's Lemma shows the claim.

\end{proof}

The next theorem contains the second main result of our work.

\begin{theorem}\label{thm:main_semiflow}
 Let $\mathbf{x}$ be a weak geometric $p$-rough path for some $p \geq 1$ and let $\omega$ be a control function which controls its $p$-variation. Assume that $\sigma = (\sigma_1,\ldots,\sigma_d)$ is a collection of $\operatorname{Lip}^{\gamma + 1}$-vector fields for some $\gamma > p$ and choose $\nu \geq |\sigma|_{\operatorname{Lip}^{\gamma}}$. Assume that $b \colon \R^m \to \R^m$ is continuous, satisfies the local conditions \eqref{eqn:bound_radial_growth_uniq}, \eqref{eqn:bound_tang_growth_uniq} on every compact set and the growth conditions \eqref{eqn:bound_radial_growth} and \eqref{eqn:bound_tang_growth}.
 
 Then the following holds true:
 \begin{enumerate}
  \item The solution semiflow $\phi$ to  \eqref{eqn:rde_with_drift} exists and is continuous.
  \item There is a constant $C$ depending on $p$, $\gamma$, $\nu$, the constants $C_1$ and $C_2$ from \eqref{eqn:bound_radial_growth} and  \eqref{eqn:bound_tang_growth} and on
  \begin{align*}
   C_3 := \sup \{ |b(\xi)|\, :\, |\xi| \leq 6 \}
  \end{align*}
  such that
  \begin{align}\label{eqn:sup_norm_bound_semiflow}
   \| \phi(0,\cdot,\xi)\|_{\infty;[0,T]} \leq C \exp(C T)(1 + N_1(\omega) + |\xi|)
  \end{align}
   and
  \begin{align}\label{eqn:p-var_bound_semiflow}
   \| \phi(0,\cdot,\xi) \|_{p-\text{var};[0,T]} \leq C \exp(C T)(1 + N_1(\omega) + |\xi|)
  \end{align}
  hold for every $\xi \in \R^m$. 
  \item For every initial condition $\xi \in \R^m$, the path $t \mapsto \phi(0,t,\xi)$ is a solution to \eqref{eqn:rde_with_drift} in the sense of Friz--Victoir.
 \end{enumerate}
 
\end{theorem}

\begin{remark}
 \begin{itemize}
  \item[(i)] For $m = 1$, the growth condition \eqref{eqn:bound_tang_growth} is always satisfied. For $m \geq 2$, assuming only condition \eqref{eqn:bound_radial_growth} does in general \emph{not} prevent explosion of the solution in finite time; see the discussion in the introduction.
  \item[(ii)] The local conditions \eqref{eqn:bound_radial_growth_uniq} and \eqref{eqn:bound_tang_growth_uniq} are satisfied in the case when $b$ is locally Lipschitz continuous. In the proof, it will become clear that they are used in order to prove uniqueness and the continuity statements. Dropping them would still imply a priori estimates for solutions to  \eqref{eqn:rde_with_drift}.
  \item[(iii)] Note that we can in general not expect to obtain a bound similar to \eqref{eqn:p-var_bound_semiflow} for the H\"older norm, not even for $\mathbf{x}$ being a H\"older rough path. Indeed, let $\sigma \equiv 0$, $m = 1$ and $b(v) = -|v|^2$. Let $y$ be the solution to  \eqref{eqn:rde_with_drift} with initial condition $\xi$. Then
  \begin{align*}
   \| y \|_{1-\text{H\"ol};[0,T]} \geq \lim_{t \searrow 0}  \frac{|y_t - y_0|}{t} = |b(\xi)| = |\xi|^2
  \end{align*}
  which shows that the H\"older norm can not grow at most linearly in $|\xi|$.

 \end{itemize}

\end{remark}

\begin{proof}[Proof of Theorem \ref{thm:main_semiflow}]
In the following proof, $C$ will be a constant which may depend on $p$, $\gamma$, $\nu$ and the constants $C_1$, $C_2$, $C_3$, but whose actual value may change from line to line. 

 As already seen in Theorem \ref{thm:flow_lin_growth}, our assumptions on $\sigma$ imply that the solution flow $\psi$ to \eqref{eqn:rough_DE_flow} exists, is twice differentiable and has a twice differentiable inverse. Define $J$ to be the inverse of its derivative and $\chi$ as in Definition \ref{def:sol_flow_rde_with_drift}. From \cite[Proposition 11]{FV10} and the chain rule,
 \begin{align*}
  \sup_{s,t \in [0,T];\, \xi \in \R^m} |D_{\xi} J(s,t,\xi)| < \infty
 \end{align*}
 and in particular, there exists a constant $C$ such that
 \begin{align*}
  \sup_{s,t \in [0,T]} |J(s,t,\xi) - J(s,t,\zeta)| \leq C |\xi - \zeta|
 \end{align*}
 holds for every $\xi, \zeta \in \R^m$. Using \ref{lemma:prop_flows}, we can choose $\delta >0$ sufficiently small such that $\omega(s,t) \leq \delta$ implies that
 \begin{align*}
  \sup_{u \in [s,t]} |J(s,u,\xi) - I_m| \leq \frac{1}{2}, \quad \sup_{u \in [s,t];\, \xi \in \R^m} |\psi(s,u,\xi) - \xi| \leq 1 
 \end{align*}
 and
 \begin{align*}
  \sup_{u \in [s,t]} |\psi(s,u,\xi) - \xi - \psi(s,u,\zeta) + \zeta| \leq \frac{1}{4} |\xi - \zeta|
 \end{align*}
 for every $\xi, \zeta \in \R^m$. Continuity of $\psi$, $J$ and $b$, Lemma \ref{lemma:a_priori_est_ODE} and Lemma \ref{lemma:uniqueness_ODE} imply that $[s,t] \ni u \mapsto \chi_s(u,\xi)$ is well defined for such $[s,t]$ and every $\xi \in \R^m$. The set
 \begin{align*}
  \mathcal{I} := \{[s,t]\,:\, \omega(s,t) \leq \delta \}
 \end{align*}
 satisfies the assumptions in Definition \ref{def:sol_flow_rde_with_drift}, and by Lemma \ref{lemma:flow_well_defined}, $\phi$ is well defined as the semiflow to \eqref{eqn:rde_with_drift}.

 We proceed with the bound for the sup-norm. Set $\tilde{\omega}(s,t) = |t - s| + \omega(s,t)$, define $(\tau_n)$ by setting $\tau_0 := 0$ and
  \begin{align*}
   \tau_{n+1} := \inf\{ u\, :\, \tilde{\omega}(\tau_n,u) \geq \delta,\, \tau_n < u \leq T \} \wedge T,
  \end{align*}
  and set 
  \begin{align*}
   N:= N_\delta(\tilde{\omega}) = \sup \{ n \in \N\,:\, \tau_n < T \}.
  \end{align*}
Choosing $\delta$ smaller if necessary, Lemma \ref{lemma:a_priori_est_ODE} shows that
  \begin{align}\label{eqn:sup_bound_interm}
   \sup_{t \in [\tau_n, \tau_{n+1}]} |\chi_{\tau_n}(t,\xi)| \leq (1 + |\xi|)\exp(C|\tau_{n+1} - \tau_n|)
  \end{align}
  holds for every $n = 0,\ldots,N$. As seen in the proof of Theorem \ref{thm:flow_lin_growth}, this bound implies that
  \begin{align*}
   \sup_{t \in [0,T]} |\phi(0,t,\xi)| \leq 1 + \exp(CT)(1 + 2N + |\xi|)
  \end{align*}
  and we can conclude as seen in Theorem \ref{thm:flow_lin_growth}. 
  
  Next, we prove the bound \eqref{eqn:p-var_bound_semiflow}. Let $\tau_n \leq u < v \leq \tau_{n+1}$. As seen in the proof of Theorem \ref{thm:flow_lin_growth}, we can use the semiflow property of $\phi$ and the triangle inequality to obtain the estimate
  \begin{align*}
   | \phi(0,v,\xi) - \phi(0,u,\xi)| \leq C \nu \omega(u, v)^{1/p} + 2| \chi_{\tau_n}(u,\phi(0,\tau_n,\xi)) - \chi_{\tau_n}(v,\phi(0,\tau_n,\xi))|.
  \end{align*}
  Using the total variation bound and the bound for the sup-norm in Lemma \ref{lemma:a_priori_est_ODE}, we see that
  \begin{align*}
   \| \phi(0,\cdot,\xi) \|_{p-\text{var};[\tau_n, \tau_{n+1}]} &\leq C \nu \delta^{1/p} + 2\| \chi_{\tau_n}(\cdot,\phi(0,\tau_n,\xi)) \|_{1-\text{var};[\tau_n,\tau_{n+1}]} \\
   &\leq C \nu \delta^{1/p} + Ce^{CT}(1 + |\phi(0,\tau_n,\xi)|)|\tau_{n+1} - \tau_n| \\ 
   &\quad + 2 (\phi(0,\tau_n,\xi) - 4)^+ - 2(\chi_{\tau_n}(\tau_{n+1},\phi(0,\tau_n,\xi)) - 4)^+.
  \end{align*}
  Therefore,
  \begin{align*}
   \| \phi(0,\cdot,\xi) \|_{p-\text{var};0, T]} &\leq \sum_{n = 0}^N  \| \phi(0,\cdot,\xi) \|_{p-\text{var};[\tau_n, \tau_{n+1}]} \\
   &\leq\ C \nu (N_{\delta}(\tilde{\omega}) + 1) \delta^{1/p} + Ce^{CT}(1 + \| \phi(0,\cdot,\xi)\|_{\infty;[0,T]}) T \\
   &\quad + 2 \sum_{n = 0}^N (\phi(0,\tau_n,\xi) - 4)^+ - (\chi_{\tau_n}(\tau_{n+1},\phi(0,\tau_n,\xi)) - 4)^+.
  \end{align*}
  For the last sum, we can estimate
  \begin{align*}
   &\sum_{n = 0}^N (\phi(0,\tau_n,\xi) - 4)^+ - (\chi_{\tau_n}(\tau_{n+1},\phi(0,\tau_n,\xi)) - 4)^+ \\
   \leq\ &|\xi| + \sum_{n = 0}^{N - 1} |\phi(0,\tau_{n+1},\xi)  - \chi_{\tau_{n}}(\tau_{n+1},\phi(0,\tau_n,\xi)) |
  \end{align*}
  and by the semiflow property of $\phi$,
  \begin{align*}
   |\phi(0,\tau_{n+1},\xi) - \chi_{\tau_{n}}(\tau_{n+1},\phi(0,\tau_n,\xi))| &= |\psi(\tau_n,\tau_{n+1},\chi_{\tau_{n}}(\tau_{n+1},\phi(0,\tau_n,\xi))) - \chi_{\tau_{n}}(\tau_{n+1},\phi(0,\tau_n,\xi))| \\
   &\leq 1
  \end{align*}
  which shows that
  \begin{align*}
   \sum_{n = 0}^N (\phi(0,\tau_n,\xi) - 4)^+ - (\chi_{\tau_n}(\tau_{n+1},\phi(0,\tau_n,\xi)) - 4)^+ \leq |\xi| + N_{\delta}(\tilde{\omega}).
  \end{align*}
  As seen in Theorem \ref{thm:flow_lin_growth}, this implies the claim.

  Next, we show that $\phi$ is continuous. As in Theorem \ref{thm:flow_lin_growth}, we can use the semiflow property to see that it is enough to prove that $[s,t] \times \R^m \ni (u,\xi) \mapsto \phi(u,v_0,\xi)$ and $[s,t] \times \R^m \ni (v,\xi) \mapsto \phi(u_0,v,\xi)$ are continuous for fixed $u_0, v_0 \in [s,t]$ where $[s,t]$ is any subinterval of $[0,T]$ with the property that $|t - s| \leq \delta$. Fix such an interval and choose sequences $u_n \to u_0$, $v_n \to v_0$ and $\xi_n \to \xi_0$ for  $n \to \infty$. We first prove that
  \begin{align*}
   \chi_{u_0}(v_n,\xi_n) \to \chi_{u_0}(v_0,\xi_0) 
  \end{align*}
  for $n \to \infty$. By the triangle inequality,
  \begin{align*}
   | \chi_{u_0}(v_n,\xi_n) - \chi_{u_0}(v_0,\xi_0) | \leq |  \chi_{u_0}(v_n,\xi_n) -  \chi_{u_0}(v_n,\xi_0) | + | \chi_{u_0}(v_n,\xi_0) - \chi_{u_0}(v_0,\xi_0) |.
  \end{align*}
  The second term converges to $0$ for $n \to \infty$ by time continuity of the solution. From Lemma \ref{lemma:a_priori_est_ODE},
  \begin{align*}
   \sup_{n \geq 0} \| \chi_{u_0}(\cdot,\xi_n) \|_{\infty;[u_0,t]} + \| \chi_{u_0}(\cdot,\xi_0) \|_{\infty;[u_0,t]} < \infty
  \end{align*}
  and therefore
  \begin{align*}
     |  \chi_{u_0}(v_n,\xi_n) -  \chi_{u_0}(v_n,\xi_0) | \leq \|  \chi_{u_0}(\cdot,\xi_n) -  \chi_{u_0}(\cdot,\xi_0) \|_{\infty;[u_0,t]} \to 0
  \end{align*}
  as $n \to \infty$ by Lemma \ref{lemma:uniqueness_ODE}. From continuity of $\psi$, this implies that $[s,t] \times \R^m \ni (v,\xi) \mapsto \phi(u_0,v,\xi)$ is continuous. Next, we use again the triangle inequality to see that
  \begin{align*}
   | \chi_{u_n}(v_0,\xi_n) - \chi_{u_0}(v_0,\xi_0) | \leq |  \chi_{u_n}(v_0,\xi_n) - \chi_{u_n}(v_0,\xi_0) | + | \chi_{u_n}(v_0,\xi_0) - \chi_{u_0}(v_0,\xi_0) |.
  \end{align*}  
  The first term converges to $0$ again by Lemma \ref{lemma:uniqueness_ODE}. It remains to show that 
  \begin{align*}
   \chi_{u_n}(v_0,\xi_0) \to  \chi_{u_0}(v_0,\xi_0)
  \end{align*}
  as $n \to \infty$. To do so, we first claim that $J(u_n,\cdot,\cdot) \to J(u_0,\cdot,\cdot)$ and $b(\psi(u_n,\cdot,\cdot)) \to b(\psi(u_0,\cdot,\cdot))$ converge uniformly on compact sets as $n \to \infty$. Indeed: For the second claim, since $b$ is continuous, it is enough to prove that $\psi(u_n,\cdot,\cdot) \to \psi(u_0,\cdot,\cdot)$ converges uniformly on compact sets as $n \to \infty$. Choose $t \in [0,T]$ and $\xi \in \R^m$. From the flow property,
  \begin{align*}
   \psi(u_n,t,\xi) = \psi(0,t,\psi(u_n,0,\xi)),
  \end{align*}
  and Lemma \ref{lemma:prop_flows} shows that it is enough to prove that $\psi(u_n,0,\cdot) \to \psi(u_0,0,\cdot)$ converges uniformly on compact sets as $n \to \infty$. By the flow property, $\psi(u_n,0,\cdot) = \psi^{-1}(0,u_n,\cdot)$, seen as homeomorphisms on $\R^m$. The inverse flow $\psi^{-1}$ is generated by a rough differential equation where we let the rough path run backwards in time (cf. \cite[Section 11.2]{FV10}), therefore the second claim follows by the standard estimates for solutions to rough differential equations, see Lemma \ref{lemma:prop_flows}. Concerning the first claim, the chain rule shows that
  \begin{align*}
   J(u_n,t,\xi) = D_{\zeta} \psi^{-1}(u_n,t,\zeta) \vert_{\zeta = \psi(u_n,t,\xi)},
  \end{align*}
  therefore it is enough to show that $D_{\xi} \psi^{-1}(u_n,\cdot,\xi) \to D_{\xi} \psi^{-1}(u_0,\cdot,\xi)$ converges uniformly on compact sets as $n \to \infty$. From $D_{\xi} \psi^{-1}(u_n,t,\xi) = D_{\xi} \psi(t,u_n,\xi)$ and the identity
  \begin{align*}
   D_{\xi} \psi(t,u_n,\xi) = D_{\zeta} \psi(0,u_n,\zeta) \vert_{\zeta = \psi(t,0,\xi)} D_{\xi} \psi(t,0,\xi),
  \end{align*}
  we see that it is sufficient to prove that $D_{\xi} \psi(0,u_n,\xi) \to D_{\xi} \psi(0,u_0,\xi)$ converges uniformly on compact sets as $n \to \infty$. Assume first that $0 \leq u_n \leq u_0$ for all $n \in \N$. Using \ref{lemma:prop_flows}, we can deduce the estimate
  \begin{align*}
   |D_{\xi} \psi(0,u_n,\xi) - D_{\xi} \psi(0,u_0,\xi)| \leq C \nu \omega(u_n,u_0)^{1/p} \exp(C \nu^p \omega(0,T)) 
  \end{align*}
  and the claim follows in this case. If $0 \leq u_0 \leq u_n$ for all $n \in \N$, the same estimate holds with $\omega(u_n,u_0)$ replaced by $\omega(u_0,u_n)$ which again implies the claim.
  
  Now set $y^n_w :=  \chi_{u_n}(w,\xi_0)$ for $n \geq 0$ and
  \begin{align*}
   R := \sup_{n \geq 0} \|y^n \|_{\infty;[u_n,t]}.
  \end{align*}
  Note that $R$ is finite by Lemma \ref{lemma:a_priori_est_ODE}. 
  We first prove right-continuity, i.e. we assume that $u_0 \leq u_n $ for all $n \geq 0$. Note that for every $n \geq 0$,
  \begin{align*}
   y^0_v = y^0_{u_n} +  \int_{u_n}^{v} J(u_0,w,y^0_w) b( \psi(u_0,w,y^0_w)) \, dw.
  \end{align*}
  By uniform convergence of $J(u_n,\cdot,\cdot)$ and $b(\psi(u_n,\cdot,\cdot))$, we can use Lemma \ref{lemma:uniqueness_ODE} to see that for any given $\varepsilon > 0$ we can choose $n$ large enough such that 
  \begin{align*}
   |y^0_{v_0} - y^n_{v_0}| \leq C(\varepsilon + |y^0_{u_n} - \xi_0|)
  \end{align*}
  holds for every $n \in \N$ where $C$ does not depend on $n$ or $\varepsilon$. By continuity of $y^0$, we see that the right hand side can be made arbitrary small for large $n$, therefore $y^n_{v_0} \to y_{v_0}$ as $n$ tends to infinity. Now we prove left-continuity, i.e. we assume that $u_n \leq u_0$. Noting that for every $n \geq 0$,
   \begin{align*}
   y^n_v = y^n_{u_0} +  \int_{u_0}^{v} J(u_n,w,y^n_w) b( \psi(u_n,w,y^n_w)) \, dw,
  \end{align*} 
  we can argue as before to see that  for any given $\varepsilon > 0$, we can find large $n$ such that
  \begin{align*}
   |y^0_{v_0} - y^n_{v_0}| \leq C(\varepsilon + |\xi_0 - y^n_{u_0}|).
  \end{align*}
  Therefore, it suffices to prove that $y^n_{u_0} \to \xi_0$ for $n \to \infty$, but this follows immediately from the estimate
  \begin{align*}
   |y^n_{u_0} - \xi_0| \leq \left| \int_{u_n}^{u_0} J(u_n,w,y^n_w) b( \psi(u_n,w,y^n_w)) \, dw \right| \leq |u_n - u_0| \sup_{u,v \in [s,t], |\xi| \leq R} |J(u,v,\xi) b( \psi(u,v,\xi))|.
  \end{align*}
  Thus we have proven left- and right continuity of $u \mapsto \chi_u(v_0,\xi_0)$ which implies that indeed $\phi$ is continuous.

  We finally show that $t \mapsto \phi(0,t,\xi)$ is a solution to \eqref{eqn:rde_with_drift} in the sense of Friz--Victoir. Fix a rough path $\mathbf{x} = \mathbf{x}^0$, an initial condition $\xi \in \R^m$ and let $(x^n)$ be a sequence of smooth paths for which the lifts $\mathbf{x}^n$ satisfy \eqref{eqn:approx_paths}. Fix some $p' \in (p,\gamma)$. Set $\omega^n(s,t) := \| \mathbf{x}^n \|_{p'-\text{var},[s,t]}^{p'}$. The $\omega^n$ are control functions which control the $p'$-variation of $\mathbf{x}^n$ for every $n \geq 0$. By interpolation \cite[Lemma 8.16]{FV10},
  \begin{align*}
   d_{p'-\text{var};[0,T]}(\mathbf{x},\mathbf{x}^n) \to 0
  \end{align*}
  as $n \to \infty$. Choose $\delta > 0$ as above. W.l.o.g., we may assume that $d_{p'-\text{var};[0,T]}(\mathbf{x},\mathbf{x}^n) \leq \delta^{1/p'}/2$ for all $n \geq 1$. Let $s \leq t$ such that $\omega^0(s,t) \leq \delta/2^{p'}$. It follows that
  \begin{align*}
   \omega^n(s,t) = \| \mathbf{x}^n \|_{p'-\text{var},[s,t]}^{p'} \leq \delta/2 + 2^{p' - 1} \| \mathbf{x} \|_{p'-\text{var},[s,t]}^{p'} \leq \delta.
  \end{align*}
  Therefore,
  \begin{align*}
   \mathcal{I}_0 := \{ [s,t]\, :\, \omega^0(s,t) \leq \delta/2^{p'} \}
  \end{align*}
  is a family of intervals for which $\omega^n(s,t) \leq \delta$ for every $[s,t] \in \mathcal{I}_0$ and every $n \geq 0$. We set $\phi^n := \phi^{\mathbf{x}^n}$ and use a similar notation for $\psi$, $J$ and $\chi$. We have to show that $\phi^n(0,\cdot,\xi) \to \phi(0,\cdot,\xi)$ uniformy as $n \to \infty$. Fix a sequence
  \begin{align*}
   0 = \tau_0 < \tau_1 < \ldots < \tau_N < \tau_{N + 1} = T
  \end{align*}
  with $[\tau_i, \tau_{i+1}] \in \mathcal{I}_{0}$ for each $i = 0,\ldots, N$. Fix $t \in [0,\tau_1]$. Then
  \begin{align*}
   |\phi(0,t,\xi) - \phi^n(0,t,\xi)| &\leq \| \psi(0,\cdot,\chi_0(\cdot,\xi)) - \psi^n (0,\cdot,\chi_0(\cdot,\xi)) \|_{\infty;[0,\tau_1]} \\
   &\quad + \|\psi^n (0,\cdot,\chi_0(t,\xi)) - \psi^n(0,\cdot,\chi^n_0(t,\xi))\|_{\infty;[0,\tau_1]}
  \end{align*}
  The first term converges to $0$ for $n \to \infty$ by \cite[Theorem 11.12]{FV10}. Concerning the second term, we can use Lemma \ref{lemma:prop_flows} to see that
  \begin{align*}
   \|\psi^n (0,\cdot,\chi_0(t,\xi)) - \psi^n(0,\cdot,\chi^n_0(t,\xi))\|_{\infty;[0,\tau_1]} \leq C |\chi_0(t,\xi) - \chi^n_0(t,\xi)|.
  \end{align*}
  Therefore it is enough to show that $\chi^n_0(\cdot,\xi) \to \chi_0(\cdot,\xi)$ as $n \to \infty$ uniformly on $[0,\tau_1]$. From \cite[Theorem 11.12 and Theorem 11.13]{FV10}, we can deduce that $J^n \to J$ and $\psi^n \to \psi$ converge uniformly as $n \to \infty$, thus the claimed convergence follows from Lemma \ref{lemma:uniqueness_ODE}. This proves that 
  \begin{align*}
   \|\phi(0,\cdot,\xi) - \phi^n(0,\cdot,\xi) \|_{\infty;[0,\tau_1]} \to 0
  \end{align*}
  as $n \to \infty$. Now assume that we have shown that
  \begin{align*}
    \|\phi(0,\cdot,\xi) - \phi^n(0,\cdot,\xi) \|_{\infty;[0,\tau_i]} \to 0
  \end{align*}
  for some $i = 1,\ldots,N$. Let $t \in [\tau_i,\tau_{i+1}]$. By the semiflow property,
  \begin{align*}
   |\phi(0,t,\xi) - \phi^n(0,t,\xi)| &= |\phi(\tau_i,t,\phi(0,\tau_i,\xi)) - \phi^n(\tau_i,t,\phi^n(0,\tau_i,\xi))| \\
   &\leq \| \psi(\tau_i,\cdot,\chi_{\tau_i}(\cdot,\phi(0,\tau_i,\xi))) - \psi^n (\tau_i,\cdot,\chi_{\tau_i}(\cdot,\phi(0,\tau_i,\xi))) \|_{\infty;[\tau_i,\tau_{i + 1}]} \\
   &\quad + \| \psi^n(\tau_i,\cdot,\chi_{\tau_i}(\cdot,\phi(0,\tau_i,\xi))) - \psi^n (\tau_i,\cdot,\chi^n_{\tau_i}(\cdot,\phi(0,\tau_i,\xi))) \|_{\infty;[\tau_i,\tau_{i + 1}]} \\
    &\quad + \| \psi^n(\tau_i,\cdot,\chi^n_{\tau_i}(\cdot,\phi(0,\tau_i,\xi))) - \psi^n (\tau_i,\cdot,\chi^n_{\tau_i}(\cdot,\phi^n(0,\tau_i,\xi))) \|_{\infty;[\tau_i,\tau_{i + 1}]}
  \end{align*}
  The first term converges to $0$ again by \cite[Theorem 11.12]{FV10}. The other two terms converge to $0$ if we can show that
  \begin{align*}
   \| \chi_{\tau_i}(\cdot,\phi(0,\tau_i,\xi)) - \chi^n_{\tau_i}(\cdot,\phi(0,\tau_i,\xi)) \|_{\infty;[\tau_i,\tau_{i + 1}]}  &\to 0 \quad \text{and} \\
   \| \chi^n_{\tau_i}(\cdot,\phi(0,\tau_i,\xi)) - \chi^n_{\tau_i}(\cdot,\phi^n(0,\tau_i,\xi)) \|_{\infty;[\tau_i,\tau_{i + 1}]}  &\to 0
  \end{align*}
  as $n \to \infty$. This follows again by using Lemma \ref{lemma:uniqueness_ODE} together with our induction hypothesis. This shows that 
  \begin{align*}
    \|\phi(0,\cdot,\xi) - \phi^n(0,\cdot,\xi) \|_{\infty;[\tau_i,\tau_{i + 1}]} \to 0
  \end{align*}
  as $n \to \infty$, and by induction hypothesis the convergence also holds uniformly on $[0,\tau_{i+1}]$ which finishes the induction. This finally proves uniform convergence on the whole time interval which shows that indeed $t \mapsto \phi(0,t,\xi)$ is a solution in the sense of Friz-Victoir.

\end{proof}

\subsection*{Acknowledgements}
\label{sec:acknowledgements}

SR would like to thank Peter Friz for valuable discussions about flow decompositions for rough differential equations. Financial support by the DFG via Research Unit FOR 2402 is gratefully acknowledged.

\bibliographystyle{alpha}
\bibliography{refs}

\end{document}